\documentclass[final,leqno,onefignum,onetabnum]{siamltexmm}
\usepackage{amsmath}
\usepackage{amsbsy}
\usepackage{amssymb}
\usepackage[dvips]{graphicx}
\usepackage{arydshln}
\newtheorem{assumption}[theorem]{Assumption}
\newtheorem{property}[theorem]{Property}
\usepackage{geometry}
\usepackage{enumerate}
\usepackage{autonum}
\newgeometry{left=1.125in,right=1.125in,top=1.3in,bottom=1.3in}

\title{A Spline Chaos Expansion
\thanks{This work was supported by the U.S. National Science Foundation under Grant Number CMMI-1607398.}}

\author{Sharif Rahman\thanks{College of Engineering and Applied Mathematics \& Computational Sciences, The University of Iowa, Iowa City, IA 52242 (\email{sharif-rahman@uiowa.edu}). Questions, comments, or corrections
to this document may be directed to that email address.}}

\begin{document}

\maketitle
\newcommand{\slugmaster}{%
\slugger{juq}{xxxx}{xx}{x}{x--x}}

\begin{abstract}
A spline chaos expansion, referred to as SCE, is introduced for uncertainty quantification analysis.  The expansion provides a means for representing an output random variable of interest with respect to multivariate orthonormal basis splines (B-splines) in input random variables.  The multivariate B-splines are built from a whitening transformation to generate univariate orthonormal B-splines in each coordinate direction, followed by a tensor-product structure to produce the multivariate version. SCE, as it stems from compactly supported B-splines, tackles locally prominent responses more effectively than the polynomial chaos expansion (PCE).  The approximation quality of the expansion is demonstrated in terms of the modulus of smoothness of the output function, leading to the mean-square convergence of SCE to the correct limit. Analytical formulae are proposed to calculate the mean and variance of an SCE approximation for a general output variable in terms of the requisite expansion coefficients.  Numerical results indicate that a low-order SCE approximation with an adequate mesh is markedly more accurate than a high-order PCE approximation in estimating the output variances and probability distributions of oscillatory, nonsmooth, and nearly discontinuous functions.
\end{abstract}


\pagestyle{myheadings}
\thispagestyle{plain}
\markboth{S. RAHMAN}{A SPLINE CHAOS EXPANSION}

\section{Introduction}
Uncertainty quantification (UQ) of complex mathematical models is a cross-cutting research topic with broad impacts on engineering and applied sciences \cite{grigoriu02,smith13,sullivan15}. A frequently employed method for UQ analysis entails polynomial chaos expansion (PCE), which describes an infinite series expansion of a square-integrable output random variable in terms of measure-consistent orthogonal polynomials in input random variables \cite{cameron47,ernst12,wiener38}.  The expansion is largely predicated on the smoothness assumption of the output function, because the polynomial basis of PCE is globally supported.  While polynomials have many attractive properties, they possess one undesirable feature: polynomials may oscillate wildly \cite{schumaker07}.  As soon as the expansion degree or order
\footnote{The nouns \emph{degree} and \emph{order} of a polynomial or spline expansion are used synonymously in the paper.}
exceeds four or five, a PCE approximation becomes prone to unstable swings. This is chiefly because polynomials are inflexible if they are too smooth, long heralded as a virtue.  They are analytic, which means that the behavior of a polynomial in an arbitrarily small region determines the behavior everywhere. In the physical world, though, the output function is frequently of a disjointed nature, meaning that the behavior in one region may be completely unrelated to the behavior in another region. In this case, the convergence property of PCE or other polynomial-based methods may become markedly deteriorated. In an effort to enhance the performance of global supported PCE, domain decomposition techniques, such as multi-element formulation of PCE, have been introduced \cite{wan05}.  However, in the presence of large subdomains of discontinuities, the multi-element PCE becomes computationally inefficient, especially when there are many input random variables.  Therefore, alternative UQ methods, proficient in tackling locally pronounced highly nonlinear or nonsmooth output functions, are desirable.

This paper presents a new, alternative orthogonal expansion, referred to as spline chaos expansion or SCE, for UQ analysis subject to independent but otherwise arbitrary probability measures of input random variables.  The paper is structured as follows.  Section 2 starts with mathematical preliminaries and assumptions.  A brief exposition of univariate basis splines (B-splines) is given in Section 3. This is followed by a presentation of orthonormal B-splines, including their second-moment properties, in Section 4.  Section 5 describes the construction of multivariate B-splines and explains how they form an orthonormal basis of a spline space of interest.  Section 6 formally presents SCE for a square-integrable random variable and then demonstrates the convergence and optimality of SCE.  The formulae for the mean and variance of an SCE approximation are derived.   The results from three numerical examples are reported in Section 7. Section 8 discusses future work. Finally, conclusions are drawn in Section 9.

\section{Input random variables}
Let $\mathbb{N}:=\{1,2,\ldots\}$, $\mathbb{N}_{0}:=\mathbb{N} \cup \{0\}$, and $\mathbb{R}:=(-\infty,+\infty)$ represent the sets of positive integer (natural), non-negative integer, and real numbers, respectively.  Denote by $[a_k,b_k]$ a finite closed interval, where $a_k, b_k \in \mathbb{R}$, $b_k>a_k$. Then, given $N \in \mathbb{N}$, $\mathbb{A}^N=\times_{k=1}^N [a_k,b_k]$ represents a closed bounded domain of $\mathbb{R}^N$.

Let $(\Omega,\mathcal{F},\mathbb{P})$ be a probability space, where $\Omega$ is a sample space representing an abstract set of elementary events, $\mathcal{F}$ is a $\sigma$-algebra on $\Omega$, and $\mathbb{P}:\mathcal{F}\to[0,1]$ is a probability measure.  Defined on this probability space, consider an $N$-dimensional input random vector $\mathbf{X}:=(X_{1},\ldots,X_{N})^\intercal$, describing the statistical uncertainties in all system parameters of a stochastic or UQ problem.  Denote by $F_{\mathbf{X}}({\mathbf{x}}):=\mathbb{P}(\cap_{i=1}^{N}\{ X_k \le x_k \})$ the joint distribution function of $\mathbf{X}$. The $k$th component of $\mathbf{X}$ is a random variable $X_k$, which has the marginal probability distribution function $F_{X_k}(x_k):=\mathbb{P}(X_k \le x_k)$.  In the UQ community, the input random variables are also known as basic random variables.  The non-zero, finite integer $N$ represents the number of input random variables and is often referred to as the dimension of the stochastic or UQ problem.

A set of assumptions on input random variables used or required by SCE is as follows.

\begin{assumption}
The input random vector $\mathbf{X}:=(X_{1},\ldots,X_{N})^\intercal$ satisfies all of the following conditions:
\begin{enumerate}[{$(1)$}]

\item
All component random variables $X_k$, $k=1,\ldots,N$, are statistically independent, but not necessarily identically distributed.

\item
Each input random variable $X_k$ is defined on a bounded interval $[a_k,b_k] \subset \mathbb{R}$.  Therefore, all moments of $X_k$ exists, that is, for all $l \in \mathbb{N}_0$,
\begin{equation}
\mathbb{E} \left[ X_k^l \right] :=
\int_{\Omega} X_k^l(\omega) d\mathbb{P}(\omega) < \infty,
\label{2.1}
\end{equation}
where $\mathbb{E}$ is the expectation operator with respect to the probability measure $\mathbb{P}$.

\item
Each input random variable $X_k$ has absolutely continuous marginal probability distribution function $F_{X_k}(x_k)$ and continuous marginal probability density function $f_{X_k}(x_k):={\partial F_{X_k}(x_k)}/{\partial x_k}$ with a bounded support $[a_k,b_k] \subset \mathbb{R}$.  Consequently, with Items (1) and (2) in mind, the joint probability distribution function $F_{\mathbf{X}}({\mathbf{x}})$ and joint probability density function $f_{\mathbf{X}}({\mathbf{x}}):={\partial^N F_{\mathbf{X}}({\mathbf{x}})}/{\partial x_1 \cdots \partial x_N}$ of $\mathbf{X}$ are obtained from
\[
F_{\mathbf{X}}({\mathbf{x}})=\prod_{k=1}^{N} F_{X_k}(x_k)~~\text{and}~~
f_{\mathbf{X}}({\mathbf{x}})=\prod_{k=1}^{N} f_{X_k}(x_k),
\]
respectively, with a bounded support $\mathbb{A}^N \subset \mathbb{R}^N$ of the density function.

\end{enumerate}
\label{a1}
\end{assumption}

Assumption \ref{a1} assures the existence of a relevant sequence of orthogonal polynomials or splines consistent with the input probability measure.  The discrete distributions and dependent variables are not dealt with in this paper.

Given the abstract probability space $(\Omega,\mathcal{F},\mathbb{P})$ of $\mathbf{X}$, there exists an image probability space $(\mathbb{A}^N,\mathcal{B}^{N},f_{\mathbf{X}}d\mathbf{x})$, where $\mathbb{A}^N$ is the image of $\Omega$ from the mapping $\mathbf{X}:\Omega \to \mathbb{A}^N$ and $\mathcal{B}^{N}:=\mathcal{B}(\mathbb{A}^{N})$ is the Borel $\sigma$-algebra on $\mathbb{A}^N \subset \mathbb{R}^N$.  Relevant statements and objects in the abstract probability space have obvious counterparts in the associated image probability space.  Both probability spaces will be exploited in this paper.

\section{Univariate B-splines}
Let $\mathbf{x}=(x_1,\ldots,x_N)$ be an arbitrary point in $\mathbb{A}^N$.  For the coordinate direction $k$, $k=1,\ldots,N$, define a positive integer $n_k \in \mathbb{N}$ and a non-negative integer $p_k \in \mathbb{N}_0$, representing the total number of basis functions and polynomial degree, respectively.  The rest of this section briefly describes paraphernalia of univariate B-splines.

\subsection{Knot sequence}
In order to define B-splines, the concept of knot sequence, also referred to as knot vector by some, for each coordinate direction $k$ is needed.
\begin{definition}
A knot sequence $\boldsymbol{\xi}_k$ for the interval $[a_k,b_k] \subset \mathbb{R}$, given $n_k > p_k \ge 0$, is a non-decreasing sequence of real numbers
\begin{equation}
\begin{array}{c}
\boldsymbol{\xi}_k:={\{\xi_{k,i_k}\}}_{i_k=1}^{n_k+p_k+1}=
\{a_k=\xi_{k,1},\xi_{k,2},\ldots,\xi_{k,n_k+p_k+1}=b_k\},   \\
\xi_{k,1} \le \xi_{k,2} \le \cdots \le \xi_{k,n_k+p_k+1}, \rule{0pt}{0.2in}
\end{array}
\label{3.1}
\end{equation}
where $\xi_{k,i_k}$ is the $i_k$th knot with $i_k=1,2,\ldots,n_k+p_k+1$ representing the knot index for the coordinate direction $k$.  The elements of $\boldsymbol{\xi}_k$ are called knots.
\label{d1}
\end{definition}

According to \eqref{3.1}, the total number of knots is $n_k+p_k+1$.  The knots may be equally spaced or unequally spaced, resulting in a uniform or non-uniform distribution. More importantly, the knots, whether they are exterior or interior, may be repeated, that is, a knot $\xi_{k,i_k}$ of the knot sequence $\boldsymbol{\xi}_k$ may appear $1 \le m_{k,i_k} \le p_k+1$ times, where $m_{k,i_k}$ is referred to as its multiplicity.  The multiplicity has important implications on the regularity properties of B-spline functions. To monitor knots without repetitions, say, there are $r_k$ distinct knots $\zeta_{k,1},\ldots,\zeta_{k,r_k}$ in $\boldsymbol{\xi}_k$ with respective multiplicities $m_{k,1},\ldots,m_{k,r_k}$.  Then the knot sequence in \eqref{3.1} can be expressed more precisely by
\begin{equation}
\begin{array}{c}
\boldsymbol{\xi}_k=\{a_k=\overset{m_{k,1}~\mathrm{times}}{\overbrace{\zeta_{k,1},\ldots,\zeta_{k,1}}},
\overset{m_{k,2}~\mathrm{times}}{\overbrace{\zeta_{k,2},\ldots,\zeta_{k,2}}},\ldots,
\overset{m_{k,r_{k}-1}~\mathrm{times}}{\overbrace{\zeta_{k,r_{k}-1},\ldots,\zeta_{k,r_{k}-1}}},
\overset{m_{k,r_{k}}~\mathrm{times}}{\overbrace{\zeta_{k,r_{k}},\ldots,\zeta_{k,r_{k}}}}=b_k\}, \\
a_k=\zeta_{k,1} < \zeta_{k,2} < \cdots < \zeta_{k,r_k-1} < \zeta_{k,r_k}=b_k,  \rule{0pt}{0.2in}
\end{array}
\label{3.2}
\end{equation}
which consists of a total number of
\[
\sum_{i_k=1}^{r_k} m_{k,i_k} = n_k+p_k+1
\]
knots. A knot sequence is called open if the end knots have multiplicities $p_k+1$.  In this case, definitions of more specific knot sequences are in order.

\begin{definition}
A knot sequence is said to be $(p_k+1)$-open if the first and last knots appear $p_k+1$ times, that is, if
\begin{equation}
\begin{array}{c}
\boldsymbol{\xi}_k=\{a_k=\overset{p_k+1~\mathrm{times}}{\overbrace{\zeta_{k,1},\ldots,\zeta_{k,1}}},
\overset{m_{k,2}~\mathrm{times}}{\overbrace{\zeta_{k,2},\ldots,\zeta_{k,2}}},\ldots,
\overset{m_{k,r_{k}-1}~\mathrm{times}}{\overbrace{\zeta_{k,r_{k}-1},\ldots,\zeta_{k,r_{k}-1}}},
\overset{p_k+1~\mathrm{times}}{\overbrace{\zeta_{k,r_{k}},\ldots,\zeta_{k,r_{k}}}}=b_k\}, \\
a_k=\zeta_{k,1} < \zeta_{k,2} < \cdots < \zeta_{k,r_k-1} < \zeta_{k,r_k}=b_k.  \rule{0pt}{0.2in}
\end{array}
\label{3.3}
\end{equation}
\label{d2}
\end{definition}

\begin{definition}
A knot sequence is said to be $(p_k+1)$-open with simple knots if it is $(p_k+1)$-open and all interior knots appear only once, that is, if
\begin{equation}
\begin{array}{c}
\boldsymbol{\xi}_k=\{a_k=\overset{p_k+1~\mathrm{times}}{\overbrace{\zeta_{k,1},\ldots,\zeta_{k,1}}},
\zeta_{k,2},\ldots,
\zeta_{k,r_{k}-1},
\overset{p_k+1~\mathrm{times}}{\overbrace{\zeta_{k,r_{k}},\ldots,\zeta_{k,r_{k}}}}=b_k\}, \\
a_k=\zeta_{k,1} < \zeta_{k,2} < \cdots < \zeta_{k,r_k-1} < \zeta_{k,r_k}=b_k.  \rule{0pt}{0.2in}
\end{array}
\label{3.4}
\end{equation}
\label{d3}
\end{definition}

A $(p_k+1)$-open knot sequence with or without simple knots is commonly found in applications \cite{cottrell09}.

\subsection{B-splines}
The B-spline functions for a given degree are defined in a recursive manner using the knot sequence as follows.

\begin{definition}
Let $\boldsymbol{\xi}_k$ be a general knot sequence of length at least $p_k+2$ for the interval $[a_k,b_k]$,
as defined by \eqref{3.1}. Denote by $B_{i_k,p_k,\boldsymbol{\xi}_k}^k(x_k)$ the $i_k$th univariate B-spline function with degree $p_k \in \mathbb{N}_0$ for the coordinate direction $k$.  Given the zero-degree basis functions,
\begin{equation}
B_{i_k,0,\boldsymbol{\xi}_k}^k(x_k) :=
\begin{cases}
1, & \xi_{k,i_k} \le x_k < \xi_{k,i_k+1}, \\
0, & \text{otherwise},
\end{cases}
\label{3.5}
\end{equation}
for $k=1,\ldots,N$, all higher-order B-spline functions on $\mathbb{R}$ are defined recursively by
\begin{equation}
B_{i_k,p_k,\boldsymbol{\xi}_k}^k(x_k) :=
\frac{x_k - \xi_{k,i_k}}{\xi_{k,i_k+p_k} - \xi_{k,i_k}} B_{i_k,p_k-1,\boldsymbol{\xi}_k}^k(x_k) +
\frac{\xi_{k,i_k+p_k+1}-x_k}{\xi_{k,i_k+p_k+1}-\xi_{k,i_k+1}} B_{i_k+1,p_k-1,\boldsymbol{\xi}_k}^k(x_k),
\label{3.6}
\end{equation}
where $1 \le k \le N$, $1 \le i_k \le n_k$, $1 \le p_k < \infty$, and $0/0$ is considered as zero.
\label{d4}
\end{definition}

The recursive formula in Definition \ref{d4} is due to Cox \cite{cox72} and de Boor \cite{deboor72}. However, a similar formula was reported by Popoviciu and Chakalov in the 1930s \cite{deboor03}.  For alternative definitions, such as those involving divided differences, readers should consult the seminal work of Schoenberg \cite{schoenberg67}.

The B-spline functions satisfy the following desirable properties \cite{cox72,deboor72,piegl97}:

\begin{property}
They are non-negative, that is, $B_{i_k,p_k,\boldsymbol{\xi}_k}^k(x_k) \ge 0$ for all $i_k$ and $x_k$.
\label{pr1}
\end{property}

\begin{property}
They are locally supported on the interval $[\xi_{k,i_k},\xi_{k,i_k+p_k+1})$ for all $i_k$.
\label{pr2}
\end{property}

\begin{property}
They are linearly independent, that is, if
\[
\displaystyle
\sum_{i_k=1}^{n_k}c_{i_k}^k B_{i_k,p_k,\boldsymbol{\xi}_k}^k(x_k)=0,
\]
then $c_{i_k}^k=0$ for all $i_k$.
\label{pr3}
\end{property}

\begin{property}
They form a partition of unity, that is,
\[
\displaystyle
\sum_{i_k=1}^{n_k} B_{i_k,p_k,\boldsymbol{\xi}_k}^k(x_k) = 1, ~x_k \in [\xi_{k,1},\xi_{k,n_k+p_k+1}].
\]
\label{pr4}
\end{property}

\begin{property}
They are pointwise $C^\infty$-continuous everywhere except at the knots $\xi_{k,i_k}$ of multiplicity $m_{k,i_k}$, where it is $C^{p_k-m_{k,i_k}}$-continuous, provided that $1 \le m_{k,i_k} < p_k+1$.
\label{pr5}
\end{property}

For an illustration, consider $k=1$, $a_1=0$, $b_1=1$, $p_1=2$, and two open knot sequences:
\[
\begin{array}{l}
(1)~\boldsymbol{\xi}_1=\{0,0,0,0.2,0.4,0.6,0.8,1,1,1\},  \\
(2)~\boldsymbol{\lambda}_1=\{0,0,0,0.2,0.4,0.6,0.6,0.8,1,1,1\}.
\end{array}
\]
Here, $\boldsymbol{\xi}_1$ is a three-open knot sequence with simple knots because the multiplicity of each interior knot is one.  In contrast, $\boldsymbol{\lambda}_1$ is merely a three-open knot sequence, as the multiplicity of the sixth knot is two. Consequently, there are seven and eight univariate quadratic B-spline basis functions for these two cases: $B_{i_1,2,\boldsymbol{\xi}_1}^1(x_1)$, $i_1=1,\ldots,7$; and $B_{i_1,2,\boldsymbol{\lambda}_1}^1(x_1)$, $i_1=1,\ldots,8$, which are illustrated in Figures \ref{fig1}(a) and \ref{fig1}(b), respectively. The basis functions for the first case are $C^1$-continuous at all interior knots, whereas the basis functions for the second case are $C^0$-continuous at $\lambda_{1,6}=\lambda_{1,7}=0.6$ and $C^1$-continuous at other interior knots.  Clearly, the regularities of B-splines depend on the degree and multiplicities of the knots selected.

\begin{figure}[htbp]
\begin{centering}
\includegraphics[scale=0.53]{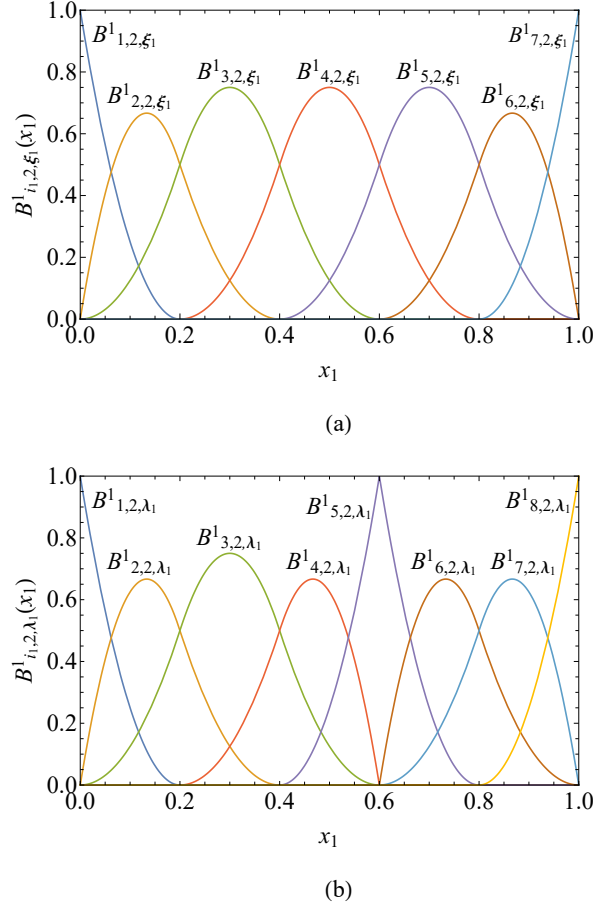}
\par\end{centering}
\caption{Quadratic B-splines generated on the interval [0,1];
(a) seven B-splines for $\boldsymbol{\xi}_1=\{0,0,0,0.2,0.4,0.6,0.8,1,1,1\}$
(b) eight B-splines for $\boldsymbol{\lambda}_1=\{0,0,0,0.2,0.4,0.6,0.6,0.8,1,1,1\}$.}
\label{fig1}
\end{figure}

\subsection{Spline space}
Suppose for $n_k > p_k \ge 0$, a knot sequence $\boldsymbol{\xi}_k$ has been specified on the interval $[a_k,b_k]$.  The associated spline space of degree $p_k$, denoted by $\mathcal{S}_{k,p_k,\boldsymbol{\xi}_k}$, is conveniently defined using an appropriate polynomial space. Define such a polynomial space as a finite-dimensional linear space
\vspace{-0.1in}
\[
\Pi_{p_k}:=
\left\{
g(x_k) =
\displaystyle
\sum_{l=0}^{p_k} c_{k,l} x_k^l: c_{k,l} \in \mathbb{R}
\right\}
\]
of real-valued polynomials in $x_k$ of degree at most $p_k$.

\begin{definition}[Schumaker \cite{schumaker07}]
For $n_k > p_k \ge 0$, let $\boldsymbol{\xi}_k$ be a
$(p_k+1)$-open knot sequence on the interval $[a_k,b_k]$, as defined by \eqref{3.3}.  Then the space
\begin{equation}
\mathcal{S}_{k,p_k,\boldsymbol{\xi}_k} :=
\left \{
\begin{array}{l}
g_k:[a_k,b_k] \to \mathbb{R}: ~\text{there exist polynomials}~g_{k,1},g_{k,2},\ldots,g_{k,r_k-1} ~\text{in}~ \Pi_{p_k}
\\
\text{such that}~g_k(x_k)=g_{k,i_k}(x_k)~\text{for}~x_k \in [\xi_{k,i_k},\xi_{k,i_k+1}),~i_k=1,\ldots,r_k-1, \\
\text{and}~
\displaystyle
\frac{\partial^{j_k} g_{k,i_k-1}}{\partial x_k}(\xi_{k,i_k}) =
\displaystyle
\frac{\partial^{j_k} g_{k,i_k}}{\partial x_k}(\xi_{k,i_k})~\text{for}~j_k=0,1,\ldots,p_k-m_{k,i_k},
\\
i_k=2,\ldots,r_k-1
\end{array}
\right\}
\label{3.7}
\end{equation}
is defined as the spline space of degree $p_k$ with distinct knots $\zeta_{k,1},\ldots,\zeta_{k,r_k}$ of multiplicities $m_{k,1}=p_k+1$, $1 \le m_{k,2} \le p_k+1$, $\ldots$, $1 \le m_{k,r_k-1} \le p_k+1$, $m_{k,r_k}=p_k+1$.
\label{d5}
\end{definition}

The spline space is uniquely determined by distinct interior knots $\zeta_{k,2},\ldots,\zeta_{k,r_k-1}$ of multiplicities $m_{k,2},\ldots,m_{k,r_k-1}$.  Indeed, the multiplicities decide the nature of $\mathcal{S}_{k,p_k,\boldsymbol{\xi}_k}$ by controlling the smoothness of the splines at interior knots.  For instance, if $m_{k,i_k}=p_k+1$, $i_k=2,\ldots,r_k-1$, then two polynomial pieces $g_{k,i_k-1}$ and $g_{k,i_k}$ in the sub-intervals adjoining the knot $\xi_{k,i_k}$ are unrelated, possibly forming a jump discontinuity at $\xi_{k,i_k}$.  In this case, $\mathcal{S}_{k,p_k,\boldsymbol{\xi}_k}$ will be the roughest space of splines.  If $m_{k,i_k}<p_k+1$, $i_k=2,\ldots,r_k-1$, then the two aforementioned polynomial pieces are connected smoothly in the sense that the first $p_k-m_{k,i_k}$ derivatives are all continuous across the knot. More specifically, if $m_{k,i_k}=1$, $i_k=2,\ldots,r_k-1$, then there are simple knots with the corresponding spline space becoming the smoothest space of piecewise polynomials of degree at most $p_k$.

\begin{proposition}[Schumaker \cite{schumaker07}]
The spline space $\mathcal{S}_{k,p_k,\boldsymbol{\xi}_k}$ is a linear space of dimension
\begin{equation}
\dim \mathcal{S}_{k,p_k,\boldsymbol{\xi}_k} = n_k = \displaystyle \sum_{i_k=2}^{r_k-1} m_{k,i_k} + p_k +1.
\label{3.8}
\end{equation}
\label{p1}
\end{proposition}

\begin{proposition}[Schumaker \cite{schumaker07}]
For $n_k > p_k \ge 0$, let $\boldsymbol{\xi}_k$ be a $(p_k+1)$-open knot sequence on the interval $[a_k,b_k]$. Denote by
\begin{equation}
\left\{
B_{1,p_k,\boldsymbol{\xi}_k}^k(x_k),\ldots,B_{n_k,p_k,\boldsymbol{\xi}_k}^k(x_k)
\right\}
\label{3.9}
\end{equation}
a set of $n_k$ B-splines of degree $p_k$. Then
\begin{equation}
\mathcal{S}_{k,p_k,\boldsymbol{\xi}_k} =
\operatorname{span}\{B_{i_k,p_k,\boldsymbol{\xi}_k}^k(x_k)\}_{i_k=1,\ldots,n_k}.
\label{3.10}
\end{equation}
\label{p2}
\end{proposition}

\section{Orthonormal B-splines}
The B-splines presented in the preceding section, although they form a basis of the spline space $\mathcal{S}_{k,p_k,\boldsymbol{\xi}_k}$, are obtained without any explicit consideration of the probability law of $X_k$.  Therefore, they are not orthogonal with respect to the probability measure $f_{X_k}(x_k)dx_k$.  A popular choice for constructing orthogonal or orthonormal basis is the Gram-Schmidt procedure \cite{golub96}. However, it is known to be ill-conditioned.  Therefore, more stable methods are needed to compute orthonormal splines consistent with the input probability measure.  In this section, a linear transformation is proposed to generate their orthonormal version. The latter splines facilitate an orthogonal series expansion in a Hilbert space, resulting in concise forms of the expansion and second-moment properties of an output random variable of interest.

\subsection{Spline moment matrix}
In reference to the set of B-splines in \eqref{3.9}, consider replacing any one of its elements with an arbitrary non-zero constant, thus creating an auxiliary set.  Without loss of generality, let
\begin{equation}
\left\{
1,B_{2,p_k,\boldsymbol{\xi}_k}^k(x_k),\ldots,B_{n_k,p_k,\boldsymbol{\xi}_k}^k(x_k)
\right\}
\label{4.1}
\end{equation}
be such a set, obtained by replacing the first element of \eqref{3.9} with 1.  Proposition \ref{p3} shows that the auxiliary B-splines are also linearly independent.

\begin{proposition}
The auxiliary set of B-splines in \eqref{4.1} is linearly independent.
\label{p3}
\end{proposition}

\begin{proof}
For constants $\bar{c}_{i_k}^k \in \mathbb{R}$, $i_k=1,\ldots,n_k$, set
\begin{equation}
\displaystyle
\bar{c}_{1}^k +
\sum_{i_k=2}^{n_k} \bar{c}_{i_k}^k B_{i_k,p_k,\boldsymbol{\xi}_k}^k(x_k) = 0.
\label{4.2}
\end{equation}
Using Property \ref{pr4}, write \eqref{4.2} as
\begin{equation}
\displaystyle
\bar{c}_{1}^k B_{1,p_k,\boldsymbol{\xi}_k}^k(x_k) +
\sum_{i_k=2}^{n_k}
\left( \bar{c}_{1}^k + \bar{c}_{i_k}^k \right) B_{i_k,p_k,\boldsymbol{\xi}_k}^k(x_k) = 0.
\label{4.3}
\end{equation}
From Property \ref{pr3},
$\{B_{1,p_k,\boldsymbol{\xi}_k}^k(x_k),\ldots,B_{n_k,p_k,\boldsymbol{\xi}_k}^k(x_k)\}$ is linearly independent, meaning that the coefficients of \eqref{4.3} must all vanish. Consequently,
\begin{equation}
\bar{c}_{i_k}^k =0, ~i_k=1,\ldots,n_k,
\label{4.4}
\end{equation}
completing the proof.
\end{proof}

When the input random variable $X_k$, instead of the real variable $x_k$, is inserted in the argument, the elements of the auxiliary set become random B-splines.  A formal definition of the spline moment matrix follows.

\begin{definition}
Let
\[
\mathbf{P}_k(X_k):=(1,B_{2,p_k,\boldsymbol{\xi}_k}^k(X_k),\ldots,B_{n_k,p_k,\boldsymbol{\xi}_k}^k(X_k))^\intercal
\]
be an $n_k$-dimensional vector of constant or random B-splines. Then the $n_k \times n_k$ matrix, defined by
\begin{equation}
\mathbf{G}_k:=\mathbb{E}[\mathbf{P}_k(X_k) \mathbf{P}_k^\intercal(X_k)],
\label{4.5}
\end{equation}
is called the spline moment matrix of $\mathbf{P}_k(X_k)$.  The matrix $\mathbf{G}_k$ exists as $X_k$ has finite moments up to order $2 p_k$, as mandated by Assumption \ref{a1}.
\label{d6}
\end{definition}

Here, any element of $\mathbf{G}_k$ represents the expectation of the product between two random splines.  However, $\mathbf{G}_k$ is not the covariance matrix of $\mathbf{P}_k(X_k)$, as the means of B-splines are not \emph{zero}.

\begin{proposition}
The spline moment matrix $\mathbf{G}_k$ is symmetric and positive-definite.
\label{p4}
\end{proposition}

\begin{proof}
By definition, $\mathbf{G}_k=\mathbf{G}_k^\intercal$.   From Proposition \ref{p3}, the elements of $\mathbf{P}_k(x_k)$ are linearly independent.  Hence, the spline moment matrix is a Gram matrix and is, therefore, positive-definite.
\end{proof}

\subsection{Whitening transformation}
From Proposition \ref{p4}, $\mathbf{G}_k$ is positive-definite and therefore invertible.  Consequently, there is a non-singular whitening matrix $\mathbf{W}_k \in \mathbb{R}^{n_k \times n_k}$ such that the factorization
\begin{equation}
\mathbf{W}_k^\intercal\mathbf{W}_k=\mathbf{G}_k^{-1}~~\text{or}~~
\mathbf{W}_k^{-1}\mathbf{W}_k^{-\intercal}=\mathbf{G}_k
\label{4.6}
\end{equation}
holds.  This leads to a set of orthonormal B-splines.

\begin{definition}
Let $\mathbf{X}:=(X_{1},\ldots,X_{N})^\intercal$ be a vector of $N\in\mathbb{N}$ input random variables fulfilling Assumption \ref{a1}.  Recall, for $n_k > p_k \ge 0$ and a specified knot sequence $\boldsymbol{\xi}_k$, that $\mathbf{P}_k(X_k)$ represents an $n_k$-dimensional vector of B-splines of degree $p_k$. Then the corresponding $n_k$-dimensional vector
\[
\boldsymbol{\psi}_k(X_k):=
(\psi_{1,p_k,\boldsymbol{\xi}_k}^k(X_k),\ldots,\psi_{n_k,p_k,\boldsymbol{\xi}_k}^k(X_k))^\intercal
\]
of orthonormal B-splines, also of degree $p_k$, is obtained from the whitening transformation
\begin{equation}
\boldsymbol{\psi}_k(X_k)= \mathbf{W}_k \mathbf{P}_k(X_k),
\label{4.7}
\end{equation}
where $\mathbf{W}_k \in \mathbb{R}^{n_k \times n_k}$ is a non-singular whitening matrix satisfying \eqref{4.6}.
\label{d7}
\end{definition}

The whitening transformation in Definition \ref{d7} is a linear transformation that converts $\mathbf{P}_k(X_k)$ into $\boldsymbol{\psi}_k(X_k)$ in such a way that the latter has uncorrelated random B-splines.  The transformation is called ``whitening" because it changes one random vector to the other, which has statistical properties akin to that of a white noise vector.  However, the condition \eqref{4.6} does not uniquely determine the whitening matrix $\mathbf{W}_k$.  There are infinitely many choices of $\mathbf{W}_k$ satisfying \eqref{4.6}.  All of these choices result in a linear transformation, decorrelating $\mathbf{P}_k(X_k)$ but producing different random vectors $\boldsymbol{\psi}_k(X_k)$ \cite{kessy16,rahman18b}.

A prominent choice for $\mathbf{W}_k$, obtained from the Cholesky factorization $\mathbf{G}_k = \mathbf{Q}_k \mathbf{Q}_k^\intercal$, is
\begin{equation}
\mathbf{W}_k=\mathbf{Q}_k^{-1},
\label{4.8}
\end{equation}
where $\mathbf{Q}_k$ is an $n_k \times n_k$ lower-triangular matrix.  The rest of the paper will use the Cholesky factorization.  Nonetheless, other whitening matrices, in conjunction with \eqref{4.6}, can be used to generate orthonormal B-splines.

\begin{proposition}
Given the preambles of Propositions \ref{p2} and \ref{p3}, the set of elements of $\boldsymbol{\psi}_k(x_k)$ from Definition \ref{d7} also spans the spline space $\mathcal{S}_{k,p_k,\boldsymbol{\xi}_k}$, that is,
\begin{equation}
\mathcal{S}_{k,p_k,\boldsymbol{\xi}_k} :=
\operatorname{span}\{\psi_{i_k,p_k,\boldsymbol{\xi}_k}^k(x_k)\}_{i_k=1,\ldots,n_k}.
\label{4.9}
\end{equation}
\label{p5}
\end{proposition}

A proof of Proposition \ref{p5} can be obtained by recognizing the elements of $\boldsymbol{\psi}_k(x_k)$ to be linearly independent.

\subsection{Statistical properties}
Similar to $\mathbf{P}_k(X_k)$, $\boldsymbol{\psi}_k(X_k)$ is also a function of random input variable $X_k$.  Proposition \ref{p6} describes its second-moment properties.

\begin{proposition}
Let $\mathbf{X}:=(X_{1},\ldots,X_{N})^\intercal:(\Omega,\mathcal{F})\to(\mathbb{A}^{N},\mathcal{B}^{N})$ be a vector of $N \in \mathbb{N}$ input random variables fulfilling Assumption \ref{a1}.  If the whitening matrix is selected as $\mathbf{Q}_k^{-1}$, then the first- and second-order moments of the vector of orthonormal B-splines $\boldsymbol{\psi}_k(X_k)= \mathbf{Q}_k^{-1} \mathbf{P}_k(X_k)$, $k=1,\ldots,N$, are
\begin{equation}
\mathbb{E} \left[ \boldsymbol{\psi}_k(X_k) \right] = (1, 0, \ldots, 0)^\intercal
\label{4.10}
\end{equation}
and
\begin{equation}
\mathbb{E} \left[ \boldsymbol{\psi}_k(X_k) \boldsymbol{\psi}_k^\intercal(X_k) \right] = \mathbf{I}_{n_k},
\label{4.11}
\end{equation}
respectively, where $\mathbf{I}_{n_k}$ is the $n_k \times n_k$ identity matrix.
\label{p6}
\end{proposition}

\begin{proof}
Using \eqref{4.8} in the whitening transformation \eqref{4.7},
\[
\begin{array}{rcl}
\mathbb{E}[\boldsymbol{\psi}_k(X_k) \boldsymbol{\psi}_k^\intercal(X_k)] & = &
\mathbf{Q}_k^{-1} \mathbb{E}[\mathbf{P}_k(X_k)\mathbf{P}_k^\intercal(X_k)] \mathbf{Q}_k^{-\intercal} \\
                                                                         & = &
\mathbf{Q}_k^{-1} \mathbf{G}_k \mathbf{Q}_k^{-\intercal} \\
                                                                         & = &
\mathbf{Q}_k^{-1}\mathbf{Q}_k \mathbf{Q}_k^{\intercal} \mathbf{Q}_k^{-\intercal} = \mathbf{I}_{n_k},
\end{array}
\]
obtaining \eqref{4.11}.  Recognize that $\psi_{1,p_k,\boldsymbol{\xi}_k}^k(X_k)$, the first element of $\boldsymbol{\psi}_k(X_k)$, is \emph{one}. Then, using \eqref{4.11}, the expectations of products between the first row of $\boldsymbol{\psi}_k(X_k)$ and all $n_k$ columns of $\boldsymbol{\psi}_k^\intercal(X_k)$ produce \eqref{4.10}.
\end{proof}

\section{Multivariate B-splines}
As the input vector $\mathbf{X}=(X_{1},\ldots,X_{N})^\intercal$ comprises independent random variables, its joint probability density function is the product of its marginal density functions.  Consequently, measure-consistent multivariate orthonormal B-splines can be easily constructed from the tensor-product of measure-consistent univariate B-splines.

\subsection{Tensor-product spline space}
For each $k=1,\ldots,N$, suppose the knot sequence $\boldsymbol{\xi}_k$ on the interval $\mathbb{A}^{\{k\}}=[a_k,b_k]$, number of basis functions $n_k$, and degree $p_k$ have been specified.  The associated vector of measure-consistent univariate orthonormal splines in $x_k$ is
\[
\boldsymbol{\psi}_k(x_k):=(\psi_{1,p_k,\boldsymbol{\xi}_k}^k(x_k),\ldots,
\psi_{n_k,p_k,\boldsymbol{\xi}_k}^k(x_k))^\intercal.
\]
Correspondingly, the spline space is $\mathcal{S}_{k,p_k,\boldsymbol{\xi}_k}$, as expressed by \eqref{3.7}.  To define tensor-product B-splines in $N$ variables and the associated spline space, define a multi-index $\mathbf{p}:=(p_1,\ldots,p_N) \in \mathbb{N}_0^N$, representing the degrees of splines in all $N$ coordinate directions.  Denote by
$\boldsymbol{\Xi}:=\{ \boldsymbol{\xi}_1,\ldots,\boldsymbol{\xi}_N \}$ a family of all $N$ knot sequences.  Because of the tensor nature of the resulting space, many properties of univariate splines carry over, described as follows.

\begin{definition}
Given $\mathbf{p}:=(p_1,\ldots,p_N)$ and $\boldsymbol{\Xi}:=\{ \boldsymbol{\xi}_1,\ldots,\boldsymbol{\xi}_N \}$, the tensor-product spline space, denoted by $\mathcal{S}_{\mathbf{p},\boldsymbol{\Xi}}$, is defined by
\begin{equation}
\mathcal{S}_{\mathbf{p},\boldsymbol{\Xi}} :=
\bigotimes_{k=1}^N \mathcal{S}_{k,p_k,\boldsymbol{\xi}_k},
\label{5.1}
\end{equation}
\label{d8}
where the symbol $\bigotimes$ stands for tensor product.
\end{definition}

It is clear from Definition \ref{d8} that $\mathcal{S}_{\mathbf{p},\boldsymbol{\Xi}}$ is a linear space of dimension $\prod_{k=1}^N n_k$. Here, $n_k$, the dimension of the spline space $\mathcal{S}_{k,p_k,\boldsymbol{\xi}_k}$, is obtained from \eqref{3.8} when each knot sequence is chosen according to \eqref{3.3}.  Each spline $g \in \mathcal{S}_{\mathbf{p},\boldsymbol{\Xi}}$ is defined on the $N$-dimensional rectangular domain
\[
\mathbb{A}^N:=\times_{k=1}^N \mathbb{A}^{\{k\}} = \times_{k=1}^N [a_k,b_k].
\]
Define two additional multi-indices $\mathbf{i}:=(i_1,\ldots,i_N) \in \mathbb{N}^N$ and $\mathbf{n}:=(n_1,\ldots,n_N) \in \mathbb{N}^N$, representing the knot indices and numbers of univariate basis functions, respectively, in all $N$ coordinate directions.  Associated with $\mathbf{i}$, define an index set
\[
\mathcal{I}_{\mathbf{n}} := \left\{ \mathbf{i}=(i_1,\ldots,i_N): 1 \le i_k \le n_k,~ k = 1,\ldots,N \right\}
\subset \mathbb{N}^N
\]
which has cardinality
\[
|\mathcal{I}_{\mathbf{n}}|=\displaystyle \prod_{k=1}^N n_k,
\]
thus matching the dimension of $\mathcal{S}_{\mathbf{p},\boldsymbol{\Xi}}$.  Then the partition defined by the knot sequences $\boldsymbol{\xi}_k$, $k=1,\ldots,N$, splits $\mathbb{A}^N$ into smaller $N$-dimensional rectangles
\[
\begin{array}{c}
\mathbb{A}_{\mathbf{i}}^N =
\left\{  \mathbf{x}: \zeta_{k,i_k} \le x_k < \zeta_{k,i_k+1}, ~k=1,\ldots,N \right\}, \\
\mathbf{i} \in \left\{ \mathbf{i}=(i_1,\ldots,i_N): 1 \le i_k \le r_k-1,~ k = 1,\ldots,N \right\} \subseteq
\mathcal{I}_{\mathbf{n}}. \rule{0pt}{0.2in}
\end{array}
\]
A mesh is defined by the partition of $\mathbb{A}^N$ into rectangular elements $\mathbb{A}_{\mathbf{i}}^N$.  Define the largest element size in each coordinate direction $k$ by
\[
h_k := \max_{1 \le l \le r_k-1} \left( \zeta_{k,l+1}-\zeta_{k,l} \right),~k=1,\ldots,N.
\]
Then, given the family of knot sequences $\boldsymbol{\Xi}=\{\boldsymbol{\xi}_1,\ldots,\boldsymbol{\xi}_N\}$,
\[
\mathbf{h}:=(h_1,\ldots,h_N) ~~\text{and}~~
h := \max_{1 \le k \le N} h_k
\]
define a vector of the largest element sizes in all $N$ coordinates and the global element size, respectively, for the domain $\mathbb{A}^N$.

\subsection{Tensor-product orthonormal B-splines}
Given the B-splines for all $N$ coordinate directions, a formal definition of tensor-product B-splines is as follows.

\begin{definition}
Let $\mathbf{X}:=(X_{1},\ldots,X_{N})^\intercal:(\Omega,\mathcal{F})\to(\mathbb{A}^{N},\mathcal{B}^{N})$ be a vector of $N \in \mathbb{N}$ input random variables fulfilling Assumption \ref{a1}.  Suppose the univariate orthonormal B-splines consistent with the marginal probability measures in all coordinate directions have been obtained as the sets $\{\psi_{1,p_k,\boldsymbol{\xi}_k}^k(x_k),\ldots,
\psi_{n_k,p_k,\boldsymbol{\xi}_k}^k(x_k)\}$, $k=1,\ldots,N$.  Then, for $\mathbf{p}=(p_1,\ldots,p_N) \in \mathbb{N}_0^N$ and  $\boldsymbol{\Xi}=\{\boldsymbol{\xi}_1,\ldots,\boldsymbol{\xi}_N\}$,
the multivariate orthonormal B-splines in $\mathbf{x}$ consistent with the probability measure $f_{\mathbf{X}}(\mathbf{x})d\mathbf{x}$ are defined as
\begin{equation}
\Psi_{\mathbf{i},\mathbf{p},\boldsymbol{\Xi}}(\mathbf{x}) :=
\displaystyle
\prod_{k=1}^N
\psi_{i_k,p_k,\boldsymbol{\xi}_k}^k(x_k),~~\mathbf{i}=(i_1,\ldots,i_N) \in \mathcal{I}_{\mathbf{n}}.
\label{5.2}
\end{equation}
\label{d9}
\end{definition}

\subsection{Statistical properties}
When the input random variables $X_1,\ldots,X_N$, instead of real variables $x_1,\ldots,x_N$, are inserted in the argument, the multivariate splines $\Psi_{\mathbf{i},\mathbf{p},\boldsymbol{\Xi}}(\mathbf{X})$, $\mathbf{i} \in \mathcal{I}_{\mathbf{n}}$, become functions of random input variables.  Therefore, it is important to establish their second-moment properties, to be exploited in Section 6.

\begin{proposition}
Let $\mathbf{X}:=(X_{1},\ldots,X_{N})^\intercal:(\Omega,\mathcal{F})\to(\mathbb{A}^{N},\mathcal{B}^{N})$ be a vector of $N \in \mathbb{N}$ input random variables fulfilling Assumption \ref{a1}.  Then the first- and second-order moments of multivariate orthonormal B-splines $\Psi_{\mathbf{i},\mathbf{p},\boldsymbol{\Xi}}(\mathbf{X})$, $\mathbf{i},\mathbf{j} \in \mathcal{I}_{\mathbf{n}}$, are
\begin{equation}
\mathbb{E} \left[ \Psi_{\mathbf{i},\mathbf{p},\boldsymbol{\Xi}}(\mathbf{X}) \right] =
\begin{cases}
1,              & \mathbf{i} =   \boldsymbol{1} :=(1,\ldots,1), \\
0,              & \mathbf{i} \ne \boldsymbol{1},
\end{cases}
\label{5.3}
\end{equation}
and
\begin{equation}
\mathbb{E} \left[ \Psi_{\mathbf{i},\mathbf{p},\boldsymbol{\Xi}}(\mathbf{X})
\Psi_{\mathbf{j},\mathbf{p},\boldsymbol{\Xi}}(\mathbf{X}) \right] =
\begin{cases}
1,              & \mathbf{i} =   \mathbf{j}, \\
0,              & \mathbf{i} \ne \mathbf{j},
\end{cases}
\label{5.4}
\end{equation}
respectively.
\label{p7}
\end{proposition}

The statistical properties of univariate orthonormal B-splines in Proposition \ref{p6},
with statistical independence in mind, lead to the result of Proposition \ref{p7}.

\subsection{Orthonormal basis}
The following proposition shows that the multivariate orthonormal splines from Definition \ref{d9} span the spline space of interest.

\begin{proposition}
Let $\mathbf{X}:=(X_{1},\ldots,X_{N})^\intercal:(\Omega,\mathcal{F})\to(\mathbb{A}^{N},\mathcal{B}^{N})$ be a vector of $N \in \mathbb{N}$ input random variables fulfilling Assumption \ref{a1}.  Then
$\{\Psi_{\mathbf{i},\mathbf{p},\boldsymbol{\Xi}}(\mathbf{x}): \mathbf{i} \in \mathcal{I}_{\mathbf{n}}\}$, the set of multivariate orthonormal B-splines for a chosen degree $\mathbf{p}$ and family of knot sequences $\boldsymbol{\Xi}$, consistent with the probability measure $f_{\mathbf{X}}(\mathbf{x})d\mathbf{x}$, is a basis of $\mathcal{S}_{\mathbf{p},\boldsymbol{\Xi}}$.  That is,
\begin{equation}
\mathcal{S}_{\mathbf{p},\boldsymbol{\Xi}} =
 \operatorname{span} \left\{ \Psi_{\mathbf{i},\mathbf{p},\boldsymbol{\Xi}}(\mathbf{x}) \right\}_{\mathbf{i} \in \mathcal{I}_{\mathbf{n}}} =
\bigotimes_{k=1}^N  \operatorname{span} \left\{ \psi_{i_k,p_k,\boldsymbol{\xi}_k}^k(x_k) \right\}_{i_k=1,\ldots,n_k},~~|\mathcal{I}_{\mathbf{n}}|=\prod_{k=1}^N n_k.
\label{5.5}
\end{equation}
\label{p8}
\end{proposition}

The statistical properties in Proposition \ref{p7} result in linear independence of the elements of $\{ \Psi_{\mathbf{i},\mathbf{p},\boldsymbol{\Xi}}(\mathbf{x})\}_{\mathbf{i} \in \mathcal{I}_{\mathbf{n}}}$.  The desired result is obtained readily.

\section{Spline chaos expansion}
Given an input random vector $\mathbf{X}:=(X_{1},\ldots,X_{N})^\intercal:(\Omega,\mathcal{F})\to(\mathbb{A}^{N},\mathcal{B}^{N})$
with the probability density function $f_{\mathbf{X}}({\mathbf{x}})$ on $\mathbb{A}^N \subset \mathbb{R}^N$, let $y(\mathbf{X}):=y(X_{1},\ldots,X_{N})$ be a real-valued, square-integrable, measurable transformation on $(\Omega, \mathcal{F})$.  Here, $y: \mathbb{A}^N \to \mathbb{R}$ represents an output function from a mathematical model, describing relevant stochastic performance of a complex system.  Associated with the image probability space $(\mathbb{A}^N,\mathcal{B}^{N},f_{\mathbf{X}}d\mathbf{x})$,
define
\[
L^2(\mathbb{A}^N,\mathcal{B}^{N},f_{\mathbf{X}}d\mathbf{x}):=
\left\{
y: \mathbb{A}^N \to \mathbb{R}:
~\int_{\mathbb{A}^N} \left| y(\mathbf{x})\right|^2  f_{\mathbf{X}}({\mathbf{x}})d\mathbf{x} < \infty
\right\}
\]
to be a weighted $L^2$-space of interest.  Clearly, $L^2(\mathbb{A}^N,\mathcal{B}^{N},f_{\mathbf{X}}d\mathbf{x})$ is a Hilbert space, which is endowed with the inner product
\[
\left( y(\mathbf{x}),z(\mathbf{x}) \right)_{L^2(\mathbb{A}^N,\mathcal{B}^{N},f_{\mathbf{X}}d\mathbf{x})}:=
\int_{\mathbb{A}^N} y(\mathbf{x})z(\mathbf{x})f_{\mathbf{X}}(\mathbf{x})d\mathbf{x}
\]
and induced norm
\[
\|y(\mathbf{x})\|_{L^2(\mathbb{A}^N,\mathcal{B}^{N},f_{\mathbf{X}}d\mathbf{x})}=
\sqrt{(y(\mathbf{x}),y(\mathbf{x}))_{L^2(\mathbb{A}^N,\mathcal{B}^{N},f_{\mathbf{X}}d\mathbf{x})}}.
\]
Similarly, for the abstract probability space $(\Omega,\mathcal{F},\mathbb{P})$, there is an isomorphic Hilbert space
\[
L^2(\Omega,\mathcal{F},\mathbb{P}):=
\left\{Y:\Omega \to \mathbb{R}: \int_{\Omega} \left|y(\mathbf{X}(\omega))\right|^2 d\mathbb{P}(\omega) < \infty \right\}
\]
of equivalent classes of output random variables $Y=y(\mathbf{X})$ with the corresponding inner product
\[
\left( y(\mathbf{X}),z(\mathbf{X}) \right)_{L^2(\Omega, \mathcal{F}, \mathbb{P})}:=
\int_{\Omega} y(\mathbf{X}(\omega))z(\mathbf{X}(\omega))d\mathbb{P}(\omega)
\]
and norm
\[
\|y(\mathbf{X})\|_{L^2(\Omega, \mathcal{F}, \mathbb{P})}:=
\sqrt{(y(\mathbf{X}),y(\mathbf{X}))_{L^2(\Omega, \mathcal{F}, \mathbb{P})}}.
\]
It is elementary to show that $y(\mathbf{X}(\omega)) \in L^2(\Omega,\mathcal{F},\mathbb{P})$ if and only if $y(\mathbf{x}) \in L^2(\mathbb{A}^N,\mathcal{B}^{N},f_{\mathbf{X}}d\mathbf{x})$.

\subsection{SCE approximation}
An SCE approximation of a square-integrable random variable $y(\mathbf{X}) \in L^2(\Omega,\mathcal{F},\mathbb{P})$ is simply its orthogonal projection onto the spline space $\mathcal{S}_{\mathbf{p},\boldsymbol{\Xi}}$, formally presented as follows.

\begin{theorem}
Let $\mathbf{X}:=(X_{1},\ldots,X_{N})^\intercal:(\Omega,\mathcal{F})\to(\mathbb{A}^{N},\mathcal{B}^{N})$ be a vector of $N \in \mathbb{N}$ input random variables fulfilling Assumption \ref{a1}.  Given a degree $\mathbf{p}$ and a family of knot sequences $\boldsymbol{\Xi}$, recall that
$\{\Psi_{\mathbf{i},\mathbf{p},\boldsymbol{\Xi}}(\mathbf{X}): \mathbf{i} \in \mathcal{I}_{\mathbf{n}}\}$ represents the set comprising multivariate orthonormal B-splines that is consistent with the probability measure $f_{\mathbf{X}}(\mathbf{x})d\mathbf{x}$.  Then, for any random variable $y(\mathbf{X}) \in L^2(\Omega, \mathcal{F}, \mathbb{P})$, there exists an orthogonal expansion in multivariate orthonormal splines in $\mathbf{X}$, referred to as an SCE approximation
\begin{equation}
y_{\mathbf{p},\boldsymbol{\Xi}}(\mathbf{X}) :=
\sum_{\mathbf{i} \in \mathcal{I}_{\mathbf{n}}} C_{\mathbf{i},\mathbf{p},\boldsymbol{\Xi}}
\Psi_{\mathbf{i},\mathbf{p},\boldsymbol{\Xi}}(\mathbf{X})
\label{6.1}
\end{equation}
of $y(\mathbf{X})$, where the SCE expansion coefficients $C_{\mathbf{i},\mathbf{p},\boldsymbol{\Xi}} \in \mathbb{R}$, $\mathbf{i} \in \mathcal{I}_{\mathbf{n}}$, are defined as
\begin{equation}
C_{\mathbf{i},\mathbf{p},\boldsymbol{\Xi}} :=
\mathbb{E} \left[  y(\mathbf{X}) \Psi_{\mathbf{i},\mathbf{p},\boldsymbol{\Xi}}(\mathbf{X}) \right] :=
\int_{\mathbb{A}^N} y(\mathbf{x})
\Psi_{\mathbf{i},\mathbf{p},\boldsymbol{\Xi}}(\mathbf{x}) f_{\mathbf{X}}(\mathbf{x})d\mathbf{x},~~
\mathbf{i} \in \mathcal{I}_{\mathbf{n}}.
\label{6.2}
\end{equation}
Furthermore, the SCE approximation is the best approximation of $y(\mathbf{X})$ in the sense that
\begin{equation}
\displaystyle
\mathbb{E}\left[ y(\mathbf{X})-y_{\mathbf{p},\boldsymbol{\Xi}}(\mathbf{X}) \right]^2 =
\inf_{g \in \mathcal{S}_{\mathbf{p},\boldsymbol{\Xi}}}
\mathbb{E}\left[ y(\mathbf{X})-g(\mathbf{X}) \right]^2,
\label{6.3}
\end{equation}
or, equivalently,
\begin{equation}
\displaystyle
{\left\| y(\mathbf{x})-y_{\mathbf{p},\boldsymbol{\Xi}}(\mathbf{x}) \right\|}_{L^2(\mathbb{A}^N,\mathcal{B}^{N},f_{\mathbf{X}}d\mathbf{x})} =
\inf_{g \in \mathcal{S}_{\mathbf{p},\boldsymbol{\Xi}}}
{\left\| y(\mathbf{X})-g(\mathbf{x}) \right\|}_{L^2(\mathbb{A}^N,\mathcal{B}^{N},f_{\mathbf{X}}d\mathbf{x})}.
\label{6.4}
\end{equation}
\label{t1}
\end{theorem}

\begin{proof}
Consider an arbitrary function $y(\mathbf{x}) \in L^2(\mathbb{A}^N,\mathcal{B}^{N},f_{\mathbf{X}}d\mathbf{x})$. Then an orthogonal projection operator $P_{\mathcal{S}_{\mathbf{p},\boldsymbol{\Xi}}}: L^2(\mathbb{A}^N,\mathcal{B}^{N},f_{\mathbf{X}}d\mathbf{x})\to \mathcal{S}_{\mathbf{p},\boldsymbol{\Xi}}$, defined by

\begin{equation}
P_{\mathcal{S}_{\mathbf{p},\boldsymbol{\Xi}}} y :=
\sum_{\mathbf{i} \in \mathcal{I}_{\mathbf{n}}} C_{\mathbf{i},\mathbf{p},\boldsymbol{\Xi}}
\Psi_{\mathbf{i},\mathbf{p},\boldsymbol{\Xi}}(\mathbf{x}),
\label{6.5}
\end{equation}
can be constructed. By definition of the random vector $\mathbf{X}$, the sequence
$\{\Psi_{\mathbf{i},\mathbf{p},\boldsymbol{\Xi}}(\mathbf{X})\}_{\mathbf{i} \in \mathcal{I}_{\mathbf{n}}}$ is a basis of the spline subspace $\mathcal{S}_{\mathbf{p},\boldsymbol{\Xi}}$ of $L^2(\Omega,\mathcal{F},\mathbb{P})$, inheriting the properties of the basis
$\{\Psi_{\mathbf{i},\mathbf{p},\boldsymbol{\Xi}}(\mathbf{x})\}_{\mathbf{i} \in \mathcal{I}_{\mathbf{n}}}$ of the spline subspace $\mathcal{S}_{\mathbf{p},\boldsymbol{\Xi}}$ of $L^2(\mathbb{A}^N,\mathcal{B}^{N},f_{\mathbf{X}}d\mathbf{x})$.
\footnote{With a certain abuse of notation, $\mathcal{S}_{\mathbf{p},\boldsymbol{\Xi}}$ is used here as a set of spline functions of both real variables ($\mathbf{x}$) and random variables ($\mathbf{X}$).}
Therefore, \eqref{6.5} leads to the expansion in \eqref{6.1}.

For deriving the expression of the expansion coefficients, define a second moment
\begin{equation}
e_{\text{SCE}}:= \mathbb{E}\Biggl[ y(\mathbf{X}) -
\sum_{\mathbf{i} \in \mathcal{I}_{\mathbf{n}}} C_{\mathbf{i},\mathbf{p},\boldsymbol{\Xi}}
\Psi_{\mathbf{i},\mathbf{p},\boldsymbol{\Xi}}(\mathbf{X})
\Biggr]^2
\label{6.6}
\end{equation}
of the difference between $y(\mathbf{X})$ and its SCE approximation.  Differentiate both sides of \eqref{6.6} with respect to $C_{\mathbf{i},\mathbf{p},\boldsymbol{\Xi}}$, $\mathbf{i} \in \mathcal{I}_{\mathbf{n}}$, to write
\begin{equation}
\begin{array}{rcl}
\displaystyle
\frac{\partial e_{\text{SCE}}}{\partial C_{\mathbf{i},\mathbf{p},\boldsymbol{\Xi}}}
& = &
\displaystyle
\frac{\partial}{\partial C_{\mathbf{i},\mathbf{p},\boldsymbol{\Xi}}}
\mathbb{E} \Biggl[  y(\mathbf{X}) - \sum_{\mathbf{j} \in \mathcal{I}_{\mathbf{n}}}
C_{\mathbf{j},\mathbf{p},\boldsymbol{\Xi}} \Psi_{\mathbf{j},\mathbf{p},\boldsymbol{\Xi}}(\mathbf{X}) \Biggl]^2
 \\
&  = &
\displaystyle
\mathbb{E} \Biggl[
\frac{\partial}{\partial C_{\mathbf{i},\mathbf{p},\boldsymbol{\Xi}}} \Biggl\{
  y(\mathbf{X}) - \sum_{\mathbf{j} \in \mathcal{I}_{\mathbf{n}}}
C_{\mathbf{j},\mathbf{p},\boldsymbol{\Xi}} \Psi_{\mathbf{j},\mathbf{p},\boldsymbol{\Xi}}(\mathbf{X}) \Biggr\}^2
\Biggr]
 \\
&  = &
\displaystyle
 2 \mathbb{E} \Biggl[
 \Biggl\{
\sum_{\mathbf{j} \in \mathcal{I}_{\mathbf{n}}}
C_{\mathbf{j},\mathbf{p},\boldsymbol{\Xi}} \Psi_{\mathbf{j},\mathbf{p},\boldsymbol{\Xi}}(\mathbf{X}) -  y(\mathbf{X}) \Biggr\}
\Psi_{\mathbf{i},\mathbf{p},\boldsymbol{\Xi}}(\mathbf{X})
 \Biggr]
 \\
&  = &
\displaystyle
 2 \Biggl\{
\sum_{\mathbf{j} \in \mathcal{I}_{\mathbf{n}}}
C_{\mathbf{j},\mathbf{p},\boldsymbol{\Xi}}
\mathbb{E} \left[ \Psi_{\mathbf{i},\mathbf{p},\boldsymbol{\Xi}}(\mathbf{X})
\Psi_{\mathbf{j},\mathbf{p},\boldsymbol{\Xi}}(\mathbf{X}) \right] -
 \mathbb{E} \left[ y(\mathbf{X}) \Psi_{\mathbf{i},\mathbf{p},\boldsymbol{\Xi}}(\mathbf{X}) \right]
\Biggr\}
 \\
&  = &
\displaystyle
 2 \Biggl\{
C_{\mathbf{i},\mathbf{p},\boldsymbol{\Xi}} -
\mathbb{E} \left[ y(\mathbf{X}) \Psi_{\mathbf{i},\mathbf{p},\boldsymbol{\Xi}}(\mathbf{X}) \right]
\Biggr\}
.
\end{array}
\label{6.7}
\end{equation}
Here, the second, third, fourth, and last lines are obtained by interchanging the differential and expectation operators, performing the differentiation, swapping the expectation and summation operators, and applying Proposition \ref{p7}, respectively. Setting ${\partial e_{\text{SCE}}}/{\partial C_{\mathbf{i},\mathbf{p},\boldsymbol{\Xi}}}=0$ in \eqref{6.7} produces the desired result in \eqref{6.2}.

Any spline function $g \in \mathcal{S}_{\mathbf{p},\boldsymbol{\Xi}}$ can be expressed by
\begin{equation}
g(\mathbf{X}) =
\sum_{\mathbf{i} \in \mathcal{I}_{\mathbf{n}}} \bar{C}_{\mathbf{i},\mathbf{p},\boldsymbol{\Xi}}
\Psi_{\mathbf{i},\mathbf{p},\boldsymbol{\Xi}}(\mathbf{X})
\label{6.8}
\end{equation}
with some real-valued coefficients $\bar{C}_{\mathbf{i},\mathbf{p},\boldsymbol{\Xi}}$, $\mathbf{i} \in \mathcal{I}_{\mathbf{n}}$.  To minimize $\mathbb{E}[ \{y(\mathbf{X})-g(\mathbf{X})\}^2]$, its derivatives with respect to the coefficients must be \emph{zero}, that is,
\begin{equation}
\displaystyle
\frac{\partial}{\partial \bar{C}_{\mathbf{i},\mathbf{p},\boldsymbol{\Xi}}}
\mathbb{E} \left[ \{y(\mathbf{X}) - g(\mathbf{X}) \}^2 \right] =
\displaystyle
\frac{\partial}{\partial \bar{C}_{\mathbf{i},\mathbf{p},\boldsymbol{\Xi}}}
\mathbb{E} \left[ \left\{y(\mathbf{X}) -
\sum_{\mathbf{i} \in \mathcal{I}_{\mathbf{n}}} \bar{C}_{\mathbf{i},\mathbf{p},\boldsymbol{\Xi}}
\Psi_{\mathbf{i},\mathbf{p},\boldsymbol{\Xi}}(\mathbf{X}) \right\}^2 \right] = 0,
~~\mathbf{i} \in \mathcal{I}_{\mathbf{n}}.
\label{6.9}
\end{equation}
From \eqref{6.7} and the following text, the derivatives are \emph{zero} only when the coefficients $\bar{C}_{\mathbf{i},\mathbf{p},\boldsymbol{\Xi}}$, $\mathbf{i} \in \mathcal{I}_{\mathbf{n}}$, match
the expansion coefficients defined in \eqref{6.2}.  Therefore, the SCE approximation is the best one, as claimed.
\end{proof}

\begin{proposition}
For any $y(\mathbf{X}) \in L^2(\Omega,\mathcal{F},\mathbb{P})$, let
$y_{\mathbf{p},\boldsymbol{\Xi}}(\mathbf{X})$ be the SCE approximation associated with a chosen degree $\mathbf{p}$ and family of knot sequences $\boldsymbol{\Xi}$.  Then the truncation error $y(\mathbf{X})-y_{\mathbf{p},\boldsymbol{\Xi}}(\mathbf{X})$ is orthogonal to the subspace $\mathcal{S}_{\mathbf{p},\boldsymbol{\Xi}} \subset L^2(\Omega,\mathcal{F},\mathbb{P})$.
\label{p9}
\end{proposition}

\begin{proof}
Let $g$ described in \eqref{6.8}, with arbitrary coefficients $\bar{C}_{\mathbf{i},\mathbf{p},\boldsymbol{\Xi}}$, $\mathbf{i} \in \mathcal{I}_{\mathbf{n}}$, be an arbitrary element of $\mathcal{S}_{\mathbf{p},\boldsymbol{\Xi}}$.  Then

\begin{equation}
\begin{array}{ll}
\displaystyle
  &
\mathbb{E}
\left[  \left\{ y(\mathbf{X}) - y_{\mathbf{p},\boldsymbol{\Xi}}(\mathbf{X}) \right\}
g(\mathbf{X}) \right] \\
= &
\mathbb{E}
\left[  \left\{ y(\mathbf{X}) -
\displaystyle
\sum_{\mathbf{j} \in \mathcal{I}_{\mathbf{n}}} C_{\mathbf{j},\mathbf{p},\boldsymbol{\Xi}} \Psi_{\mathbf{j},\mathbf{p},\boldsymbol{\Xi}}(\mathbf{X}) \right\}
\displaystyle
\sum_{\mathbf{i} \in \mathcal{I}_{\mathbf{n}}} \bar{C}_{\mathbf{i},\mathbf{p},\boldsymbol{\Xi}}
\Psi_{\mathbf{i},\mathbf{p},\boldsymbol{\Xi}}(\mathbf{X})
\right]
\\
= &
\displaystyle
\sum_{\mathbf{i} \in \mathcal{I}_{\mathbf{n}}} C_{\mathbf{i},\mathbf{p},\boldsymbol{\Xi}} \bar{C}_{\mathbf{i},\mathbf{p},\boldsymbol{\Xi}} -
\sum_{\mathbf{i} \in \mathcal{I}_{\mathbf{n}}} C_{\mathbf{i},\mathbf{p},\boldsymbol{\Xi}} \bar{C}_{\mathbf{i},\mathbf{p},\boldsymbol{\Xi}}
\\
= &
0,
\end{array}
\label{6.10}
\end{equation}
where the third line follows from \eqref{6.2} and Proposition \ref{p7}.  Hence, the proposition is proved.
\end{proof}

\begin{proposition}
The projection operator $P_{\mathcal{S}_{\mathbf{p},\boldsymbol{\Xi}}}:L^2(\mathbb{A}^N,\mathcal{B}^{N},f_{\mathbf{X}}d\mathbf{x}) \to \mathcal{S}_{\mathbf{p},\boldsymbol{\Xi}}$
is a linear, bounded operator.
\label{p10}
\end{proposition}

\begin{proof}
The operator $P_{\mathcal{S}_{\mathbf{p},\boldsymbol{\Xi}}}$ is obviously linear.  To prove its boundedness, use Proposition \ref{p9} and then invoke the Pythagoras theorem, yielding
\begin{equation}
\mathbb{E}[\{y(\mathbf{X})-y_{\mathbf{p},\boldsymbol{\Xi}}(\mathbf{X})\}^2] + \mathbb{E}[y_{\mathbf{p},\boldsymbol{\Xi}}^2(\mathbf{X})] =
\mathbb{E}[y^2(\mathbf{X})].
\label{6.11}
\end{equation}
Therefore,
\begin{equation}
\mathbb{E}[y_{\mathbf{p},\boldsymbol{\Xi}}^2(\mathbf{X})] \le
\mathbb{E}[y^2(\mathbf{X})]
\label{6.12}
\end{equation}
for any $y(\mathbf{X}) \in L^2(\Omega, \mathcal{F}, \mathbb{P}$.  This is equivalent to the assertion that
\begin{equation}
{\left\| P_{\mathcal{S}_{\mathbf{p},\boldsymbol{\Xi}}} y(\mathbf{x}) \right\|}_{L^2(\mathbb{A}^N,\mathcal{B}^{N},f_{\mathbf{X}}d\mathbf{x})} \le
{\left\| y(\mathbf{x}) \right\|}_{L^2(\mathbb{A}^N,\mathcal{B}^{N},f_{\mathbf{X}}d\mathbf{x})}
\label{6.13}
\end{equation}
for any $y(\mathbf{x}) \in L^2(\mathbb{A}^N,\mathcal{B}^{N},f_{\mathbf{X}}d\mathbf{x})$.
\end{proof}

From the general properties of orthogonal projection, the proofs of Theorem \ref{t1} and Propositions \ref{p9} and \ref{p10} are straightforward and may deem unnecessary to the eye of an expert reader.  Nonetheless, they are documented here for the paper to be self-contained.

\subsection{Approximation quality and convergence}
A preferred approach among approximation theorists to measure the quality of approximations by polynomials and splines involves the modulus of smoothness \cite{dahmen80,schumaker07,timan63}.  Formal definitions of the modulus of smoothness in each coordinate direction $k$, followed by a tensorized version, are presented as follows.

\begin{definition}[Schumaker \cite{schumaker07}]
Given a positive integer $\alpha_k \in \mathbb{N}$ and $0 < h_k \le (b_k-a_k)/\alpha_k$, the $\alpha_k$th modulus of smoothness of a function $y(x_k) \in L^2[a_k,b_k]$ in the $L^2$-norm is a function defined by
\begin{equation}
\omega_{\alpha_k}(y;h_k)_{L^2[a_k,b_k]} := \sup_{ 0 \le u_k \le h_k}
\left\| \Delta_{u_k}^{\alpha_k} y(x_k)\right\|_{L^2[a_k,b_k-\alpha_k u_k]},~~h_k > 0,
\label{6.14}
\end{equation}
where
\[
\Delta_{u_k}^{\alpha_k} y(x_k) := \sum_{i=0}^{\alpha_k} (-1)^{\alpha_k-i} \binom{\alpha_k}{i} y(x_k+iu_k)
\]
is the $\alpha_k$th forward difference of $y$ at $x_k$ for any $0 \le u_k \le h_k$.

Moreover, given a multi-index $\boldsymbol{\alpha}=(\alpha_1,\ldots,\alpha_N) \in \mathbb{N}^N$ and any vector $\mathbf{u} \ge \boldsymbol{0}$, let
\[
\Delta_{\mathbf{u}}^{\boldsymbol{\alpha}} = \prod_{k=1}^N \Delta_{u_k}^{\alpha_k}.
\]
Then the $\boldsymbol{\alpha}$-modulus of smoothness of a function $y(\mathbf{x}) \in L^2(\mathbb{A}^N)$ in the $L^2$-norm is the function defined by
\begin{equation}
\omega_{\boldsymbol{\alpha}}(y;\mathbf{h})_{L^2(\mathbb{A}^N)} :=
\sup_{ \boldsymbol{0} \le \mathbf{u} \le \mathbf{h}}
\left\| \Delta_{\mathbf{u}}^{\boldsymbol{\alpha}} y(\mathbf{x}) \right\|_{L^2(\mathbb{A}_{\boldsymbol{\alpha},\mathbf{u}}^N)},~~
\mathbf{h} > \boldsymbol{0},
\label{6.15}
\end{equation}
where
\[
\mathbb{A}_{\boldsymbol{\alpha},\mathbf{u}}^N = \left\{  \mathbf{x} \in \mathbb{A}^N:
\mathbf{x} + \boldsymbol{\alpha}\otimes\mathbf{u} \in \mathbb{A}^N \right\},~~
\boldsymbol{\alpha}\otimes\mathbf{u} = (\alpha_1 u_1,\ldots,\alpha_N u_N).
\]
\label{d10}
\end{definition}

The book by Schumaker \cite{schumaker07} provides a slightly general definition of the modulus of smoothness for $y \in L^q[a_k,b_k]$ (Chapter 2) or $y \in L^q(\mathbb{A}^N)$ (Chapter 13), $1 \le q < \infty$, including a summary of their elementary properties.

From Definition \ref{d10}, as $h_k$ approaches \emph{zero}, so does $0 \le u_k \le h_k$.  Taking the limit $u_k \to 0$ inside the integral of the $L^2$ norm, which is permissible for a finite interval and uniformly convergent integrand, the forward difference
\[
\displaystyle
\lim_{u_k \to 0}
\Delta_{u_k}^{\alpha_k} y(x_k) =
y(x_k) \sum_{i=0}^{\alpha_k} (-1)^{\alpha_k-i} \binom{\alpha_k}{i} = 0,
\]
as the sum vanishes for any $\alpha_k \in \mathbb{N}$.  Consequently, the coordinate modulus of smoothness
\[
\omega_{\alpha_k}(y;h_k)_{L^2[a_k,b_k]} \to 0 ~~\text{as}~ h_k \to 0~~\forall \alpha_k \in \mathbb{N}.
\]
Following similar considerations, the tensor modulus of smoothness
\[
\omega_{\boldsymbol{\alpha}}(y;\mathbf{h})_{L^2(\mathbb{A}^N)} \to 0 ~~\text{as}~\mathbf{h} \to \boldsymbol{0}~~\forall \boldsymbol{\alpha} \in \mathbb{N}^N.
\]
These limits, in conjunction with Lemma \ref{l1}, will be used to prove the $L^2$-convergence of the SCE approximations.

\begin{lemma}
Let $L^2(\mathbb{A}^N)$ be an unweighted Hilbert space, defined as
\begin{equation}
L^2\left( \mathbb{A}^N \right):=
\left\{ y:\mathbb{A}^N \to \mathbb{R}:
\int_{\mathbb{A}^N} |y(\mathbf{x})|^2 d \mathbf{x} < \infty \right\},
\label{6.16}
\end{equation}
with standard norm $\| \cdot \|_{L^2(\mathbb{A}^N)}$.  Then, for any function $y(\mathbf{x}) \in L^2(\mathbb{A}^N,\mathcal{B}^{N},f_{\mathbf{X}}d\mathbf{x})$, it holds that
\begin{equation}
\left\|  y(\mathbf{x}) \right\|_{L^2(\mathbb{A}^N,\mathcal{B}^{N},f_{\mathbf{X}}d\mathbf{x})}  \le
\sqrt{\left\|  f_{\mathbf{X}}(\mathbf{x})\right\|_{L^\infty(\mathbb{A}^N)}}
\left\|  y(\mathbf{x}) \right\|_{L^2(\mathbb{A}^N)},
\label{6.17}
\end{equation}
where $\| \cdot \|_{L^\infty(\mathbb{A}^N)}$ is the infinity norm.  Here, additionally, it is assumed that $f_{\mathbf{X}} \in L^\infty(\mathbb{A}^N)$.
\label{l1}
\end{lemma}

\begin{proof}
From definition,
\begin{equation}
\begin{array}{rcl}
\left\|  y(\mathbf{x}) \right\|_{L^2(\mathbb{A}^N,\mathcal{B}^{N},f_{\mathbf{X}}d\mathbf{x})}^2
& :=   &
\int_{\mathbb{A}^N} y^2(\mathbf{x}) f_{\mathbf{X}}(\mathbf{x})d\mathbf{x}  \\
&  =   &
\left( y^2(\mathbf{x}),f_{\mathbf{X}}(\mathbf{x}) \right)_{L^2(\mathbb{A}^N)} \\
&  \le &
\left\|  y^2(\mathbf{x}) \right\|_{L^1(\mathbb{A}^N)} \cdot
\left\|  f_{\mathbf{X}}(\mathbf{x}) \right\|_{L^\infty(\mathbb{A}^N)} \\
&  =   &
\left\|  y(\mathbf{x}) \right\|_{L^2(\mathbb{A}^N)}^2 \cdot
\left\|  f_{\mathbf{X}}(\mathbf{x}) \right\|_{L^\infty(\mathbb{A}^N)}
\end{array}
\label{6.18}
\end{equation}
where the third line stems from H\"{o}lder's inequality.  As $\left\|  f_{\mathbf{X}}(\mathbf{x})\right\|_{L^\infty(\mathbb{A}^N)}$ is positive, applying the square-root on \eqref{6.18} yields the desired result.
\end{proof}

\begin{proposition}
For any $y(\mathbf{X}) \in L^2(\Omega, \mathcal{F}, \mathbb{P})$, a sequence of SCE approximations
$\{y_{\mathbf{p},\boldsymbol{\Xi}}(\mathbf{X})\}_{\mathbf{h} > \boldsymbol{0}}$, with $\mathbf{h}=(h_1,\ldots,h_N)$ representing the vector of largest element sizes, converges to $y(\mathbf{X})$ in mean-square, that is,
\[
\lim_{\mathbf{h} \to \boldsymbol{0}} \mathbb{E}\left[ \left|y(\mathbf{X})-y_{\mathbf{p},\boldsymbol{\Xi}}(\mathbf{X})\right|^2 \right] = 0.
\]
Furthermore, the sequence of SCE approximations converges in probability, that is, for any $\epsilon>0$,
\[
\lim_{\mathbf{h} \to \boldsymbol{0}} \mathbb{P}\left( \left|y(\mathbf{X})-y_{\mathbf{p},\boldsymbol{\Xi}}(\mathbf{X})\right| > \epsilon \right) = 0;
\]
and converges in distribution, that is, for all points $\xi \in \mathbb{R}$ where $F(\xi)$ is continuous,
\[
\lim_{\mathbf{h} \to \boldsymbol{0}} F_{\mathbf{p},\boldsymbol{\Xi}}(\xi) = F(\xi)
\]
such that $F_{\mathbf{p},\boldsymbol{\Xi}}(\xi):=\mathbb{P}(y_{\mathbf{p},\boldsymbol{\Xi}}(\mathbf{X}) \le \xi)$ and $F(\xi):=\mathbb{P}(y(\mathbf{X}) \le \xi)$ are distribution functions of $y_{\mathbf{p},\boldsymbol{\Xi}}(\mathbf{X})$ and $y(\mathbf{X})$, respectively. If $F(\xi)$ is continuous on $\mathbb{R}$, then the distribution functions converge uniformly.
\label{p11}
\end{proposition}

\begin{proof}
From Lemma \ref{l1},
\begin{equation}
\left\|  y(\mathbf{x}) - y_{\mathbf{p},\boldsymbol{\Xi}}(\mathbf{x})
\right\|_{L^2(\mathbb{A}^N,\mathcal{B}^{N},f_{\mathbf{X}}d\mathbf{x})}  \le
\sqrt{\left\|  f_{\mathbf{X}}(\mathbf{x})\right\|_{L^\infty}}
\left\|  y(\mathbf{x}) - y_{\mathbf{p},\boldsymbol{\Xi}}(\mathbf{x}) \right\|_{L^2(\mathbb{A}^N)}.
\label{6.19}
\end{equation}
Recognize from Proposition \ref{p10} that $P_{\mathcal{S}_{\mathbf{p},\boldsymbol{\Xi}}}$ is a linear, bounded operator.  Therefore, invoke Theorem 12.8 of Schumaker's book \cite{schumaker07}, which states that for a bounded linear operator, the unweighted $L^2$-error from the SCE approximation is bounded by
\begin{equation}
\left\| y(\mathbf{x}) - y_{\mathbf{p},\boldsymbol{\Xi}}(\mathbf{x}) \right\|_{L^2(\mathbb{A}^N)}
\le C' \omega_{\mathbf{p}+\boldsymbol{1}}(y;\mathbf{h})_{L^2(\mathbb{A}^N)},
\label{6.20}
\end{equation}
where $C'$ is a constant that depends only on $\mathbf{p}$ and $N$, and
$\mathbf{p}+\boldsymbol{1} = (p_1+1,\ldots,p_N+1)$.  Combining \eqref{6.19} and \eqref{6.20} produces
\begin{equation}
\left\|  y(\mathbf{x}) - y_{\mathbf{p},\boldsymbol{\Xi}}(\mathbf{x})
\right\|_{L^2(\mathbb{A}^N,\mathcal{B}^{N},f_{\mathbf{X}}d\mathbf{x})}  \le
C \omega_{\mathbf{p}+\boldsymbol{1}}(y;\mathbf{h})_{L^2(\mathbb{A}^N)},
\label{6.21}
\end{equation}
where $C = C'\sqrt{\left\|  f_{\mathbf{X}}(\mathbf{x})\right\|_{L^\infty}}$ is another constant, depending on $\mathbf{p}$, $N$, and now $f_{\mathbf{X}}(\mathbf{x})$.

Equation \eqref{6.21} gives a result on the $L^2$-distance of a function $y$ to the spline space $\mathcal{S}_{\mathbf{p},\boldsymbol{\Xi}}$.  From the discussion related to Definition \ref{d10}, the modulus of smoothness
\[
\omega_{\mathbf{p}+\boldsymbol{1}}(y;\mathbf{h})_{L^2(\mathbb{A}^N)} \to 0 ~~\text{as}~\mathbf{h} \to \boldsymbol{0}~~\forall \mathbf{p} \in \mathbb{N}_0^N.
\]
Therefore,
\begin{equation}
\displaystyle
\lim_{\mathbf{h} \to \boldsymbol{0}}
\left\|  y(\mathbf{x}) - y_{\mathbf{p},\boldsymbol{\Xi}}(\mathbf{x})
\right\|_{L^2(\mathbb{A}^N,\mathcal{B}^{N},f_{\mathbf{X}}d\mathbf{x})}  = 0,
\label{6.22}
\end{equation}
thus proving the mean-square convergence of $y_{\mathbf{p},\boldsymbol{\Xi}}(\mathbf{X})$ to $y(\mathbf{X})$ for any degree $\mathbf{p} \in \mathbb{N}_0^N$.  In addition, as the SCE approximation converges in mean-square, it does so in probability.  Moreover, as the expansion converges in probability, it also converges in distribution.
\end{proof}

\subsection{A special case of SCE}
The well-known PCE approximation, especially its tensor-product version, can be derived from the SCE approximation proposed.

\begin{proposition}
Given $k=1,\ldots,N$, $0 \le p_k < \infty$, and an interval $[a_k,b_k] \subset \mathbb{R}$, let
\begin{equation}
\boldsymbol{\xi}_k^{'}=\{\overset{p_k+1~\mathrm{times}}{\overbrace{a_k,\ldots,a_k}},
\overset{p_k+1~\mathrm{times}}{\overbrace{b_k,\ldots,b_k}}\}
\label{6.23}
\end{equation}
be a $(p_k+1)$-open knot sequence with no internal knots and   $\boldsymbol{\Xi}^{'}=\{\boldsymbol{\xi}_1^{'},\ldots,\boldsymbol{\xi}_N^{'}\}$.  Then the resulting SCE approximation reduces to a PCE approximation.

\label{p12}
\end{proposition}

For the knot sequence $\boldsymbol{\xi}_k^{'}$ in \eqref{6.23}, the resulting B-splines are related to the well-known Bernstein polynomials of degree $p_k$.  Since the set of Bernstein polynomials of degree $p_k$ forms a basis of the polynomial space $\Pi_{p_k}$, the spline space $\mathcal{S}_{k,p_k,\boldsymbol{\xi}_k^{'}}=\Pi_{p_k}$.  Then, going through the standard tensor-product construction, it is trivial to show that, indeed, the multivariate spline space $\mathcal{S}_{\mathbf{p},\boldsymbol{\Xi}^{'}}$ is spanned by the set $\{ \Psi_{\mathbf{i}}(\mathbf{x}): \boldsymbol{0} \le \mathbf{i} \le \mathbf{p}\}$ of multivariate orthonormal polynomials in $\mathbf{x}$ that are consistent with the probability measure $f_{\mathbf{X}}(\mathbf{x}) d\mathbf{x}$. This results in a $\mathbf{p}$th-degree tensor-product PCE
\begin{equation}
y_{\mathbf{p},\boldsymbol{\Xi}^{'}}(\mathbf{X}) =
y_{\mathbf{p}}(\mathbf{X}) :=
\sum_{\boldsymbol{0} \le \mathbf{i} \le \mathbf{p}} {C}_{\mathbf{i}}^{'}
\Psi_{\mathbf{i}}(\mathbf{X}),~~
\label{6.24}
\end{equation}
of $y(\mathbf{X}) \in L^2(\Omega,\mathcal{F},\mathbb{P})$, where $\mathbf{p} \in \mathbb{N}_0^N$ and
\[
{C}_{\mathbf{i}}^{'} :=
\int_{\mathbb{A}^N} y(\mathbf{x})
\Psi_{\mathbf{i}}(\mathbf{x}) f_{\mathbf{X}}(\mathbf{x})d\mathbf{x},~~
\mathbf{i} \in \mathbb{N}_0^N,
\]
are its expansion coefficients.  Hence, a $\mathbf{p}$th-degree SCE approximation with no internal knots becomes identical to a $\mathbf{p}$th-degree PCE approximation.

As described in the preceding paragraph, there is no mesh in the deduction of the PCE approximation from the SCE approximation.  Therefore, a refinement by reducing the element sizes is not possible.
However, a refinement is still possible by degree elevation, that is, by increasing gradually the degree $p_k$ in all coordinate directions. Indeed, when $p_k \to \infty$ for all $k=1,\ldots,N$, the right side of \eqref{6.24} becomes the full PCE representation of $y(\mathbf{X})$.

\subsection{Output statistics and other probabilistic characteristics}
The SCE approximation $y_{\mathbf{p},\boldsymbol{\Xi}}(\mathbf{X})$ can be viewed as a surrogate of $y(\mathbf{X})$.  Therefore, relevant probabilistic characteristics of $y(\mathbf{X})$, including its first two moments and probability density function, if it exists, can be estimated from the statistical properties of $y_{\mathbf{p},\boldsymbol{\Xi}}(\mathbf{X})$.

Applying the expectation operator on $y_{\mathbf{p},\boldsymbol{\Xi}}(\mathbf{X})$ in \eqref{6.1} and imposing Proposition \ref{p7}, its mean is
\begin{equation}
\mathbb{E}\left[ y_{\mathbf{p},\boldsymbol{\Xi}}(\mathbf{X}) \right] = C_{\boldsymbol{1},\mathbf{p},\boldsymbol{\Xi}} =
\mathbb{E}\left[ y(\mathbf{X})\right],~\boldsymbol{1}=(1,\ldots,1),
\label{6.25}
\end{equation}
which is independent of $\mathbf{p}$ and $\boldsymbol{\Xi}$. More importantly, the SCE approximation always yields the exact mean.

Applying the expectation operator again, this time on $[y_{\mathbf{p},\boldsymbol{\Xi}}(\mathbf{X})-C_{\boldsymbol{1},\mathbf{p},\boldsymbol{\Xi}}]^2$, and employing Proposition \ref{p7} one more time results in the variance
\begin{equation}
\operatorname{var}\left[ y_{\mathbf{p},\boldsymbol{\Xi}}(\mathbf{X}) \right] =
\sum_{\mathbf{i} \in (\mathcal{I}_{\mathbf{n}}\setminus{\{\boldsymbol{1}\}})} C_{\mathbf{i},\mathbf{p},\boldsymbol{\Xi}}^2 =
\sum_{\mathbf{i} \in \mathcal{I}_{\mathbf{n}}} C_{\mathbf{i},\mathbf{p},\boldsymbol{\Xi}}^2 -  C_{\boldsymbol{1},\mathbf{p},\boldsymbol{\Xi}}^2
\label{6.26}
\end{equation}
of $y_{\mathbf{p},\boldsymbol{\Xi}}(\mathbf{X})$. It is elementary to show that $\operatorname{var}[ y_{\mathbf{p},\boldsymbol{\Xi}}(\mathbf{X})] \le \operatorname{var}[ y(\mathbf{X})]$.

The second-moment properties of an SCE approximation are solely determined by the expansion coefficients.  The formulae for the mean and variance of the SCE approximation are same as those reported for the PCE approximation, although the respective expansion coefficients involved are not.  The primary reason for this similarity stems from the use of orthonormal basis in both expansions.

Being convergent in probability and in distribution, the probability density function of $y(\mathbf{X})$, if it exists, can also be estimated by that of $y_{\mathbf{p},\boldsymbol{\Xi}}(\mathbf{X})$.  However, deriving analytical formula for the density function is hopeless in general.  Nonetheless, the density function can be estimated by Monte Carlo simulation of the SCE approximation, that is, by re-sampling of $y_{\mathbf{p},\boldsymbol{\Xi}}(\mathbf{X})$ involving inexpensive evaluations of simple spline functions.

\subsection{SCE as an infinite series}
The set of orthonormal B-splines $\{\Psi_{\mathbf{i},\mathbf{p},\boldsymbol{\Xi}}(\mathbf{x}): \mathbf{i} \in \mathcal{I}_{\mathbf{n}}\}$ from \eqref{6.1} has its size equal to $\prod_{k=1}^N n_k$.  Therefore, the size is controlled by the number of basis functions $n_k$, which, in succession, is decided by the length of the knot sequence $\boldsymbol{\xi}_k$ and order $p_k$ in each coordinate direction. Obviously, the longer the sequence $\boldsymbol{\xi}_k$, the larger the value of $n_k$ and, hence, the size of the set. For a refinement process with a fixed $p_k$, consider increasing the length of $\boldsymbol{\xi}_k$ or $n_k$ in all $N$ coordinate directions in such a way that the largest element size $h_k$ is monotonically reduced. The result is an increasing family of the sets of such basis functions. In the limit, when $n_k \to \infty$ or $h_k \to 0$, $k=1,\ldots,N$, denote by $\boldsymbol{\xi}_{k,\infty}$ and $\boldsymbol{\Xi}_\infty=\{\boldsymbol{\xi}_{1,\infty},\ldots,\boldsymbol{\xi}_{N,\infty}\}$ the
associated knot sequence in the $k$th coordinate direction and the family of such $N$ knot sequences, respectively.  Then there exists a set of infinite number of basis functions  $\{\Psi_{\mathbf{i},\mathbf{p},\boldsymbol{\Xi}_\infty}(\mathbf{x}): \mathbf{i} \in \mathbb{N}^N\}$ with the index set of knot indices
\[
\left\{ \mathbf{i}=(i_1,\ldots,i_N): 1 \le i_k < \infty,~ k = 1,\ldots,N \right\}
= \mathbb{N}^N.
\]
In consequence, $\{\Psi_{\mathbf{i},\mathbf{p},\boldsymbol{\Xi}_\infty}(\mathbf{x}): \mathbf{i} \in \mathbb{N}^N\}$ forms an orthogonal basis of $\mathcal{S}_{\mathbf{p},\boldsymbol{\Xi}_\infty}$, yielding
\[
L^2(\Omega,\mathcal{F},\mathbb{P}) =
\overline{
\text{span}\{\Psi_{\mathbf{i},\mathbf{p},\boldsymbol{\Xi}_\infty}(\mathbf{x})\}_{\mathbf{i} \in \mathbb{N}^N}
},
\]
where the overline stands for set closure.  Hence, every $y(\mathbf{X}) \in L^2(\Omega,\mathcal{F},\mathbb{P})$ can be expanded in terms of the aforementioned spanning set, resulting in an infinite series
\begin{equation}
y(\mathbf{X}) \sim
\sum_{\mathbf{i} \in \mathbb{N}^N} C_{\mathbf{i},\mathbf{p},\boldsymbol{\Xi}_\infty}
\Psi_{\mathbf{i},\mathbf{p},\boldsymbol{\Xi}_\infty}(\mathbf{X}),
\label{6.27}
\end{equation}
which is referred to as SCE in the paper. Here the symbol $\sim$ represents equality in the mean-square sense.  From Proposition \ref{p11}, the partial sums of \eqref{6.27} converge to $y(\mathbf{X})$ in $L^2$. Therefore, the infinite series is mean-square convergent to the correct limit.

Comparing \eqref{6.1} and \eqref{6.27}, the former is a truncated version of the latter.  Therefore, the designation ``SCE" employed in the paper makes sense even though the SCE approximation in \eqref{6.1} represents a finite sum or expansion.

\section{Numerical experiments}
Three examples describing one-, two-, and four-dimensional UQ problems, where the output function is explicitly defined or obtained from the solution of an ordinary differential equation (ODE), are presented. The random input $\mathbf{X}$ fulfills Assumption \ref{a1}, and the output function $y(\mathbf{X})$ is in $L^2(\Omega,\mathcal{F},\mathbb{P})$. Therefore, SCE and PCE approximations can be applied to estimate their second-moment statistics and probability distributions.  The objectives are to evaluate the approximation power of the SCE approximation in terms of the second-moment statistics or probability distributions of $y(\mathbf{X})$ and contrast the SCE results with those obtained from the existing PCE approximation.

The coordinate degrees for SCE or PCE approximations in the second and third examples are identical, that is, $p_1=p_2=p_3=p_4=p$ (say).  So are the knot sequences for SCE, that is, $\boldsymbol{\xi}_1=\boldsymbol{\xi}_2=\boldsymbol{\xi}_3=\boldsymbol{\xi}_4=\boldsymbol{\xi}$ (say) with a uniform mesh of element sizes $h_1=h_2=h_3=h_4=h$.  In all three examples, the degree $p$ and/or element size $h$ were varied as desired.  The basis for a $p$th-degree PCE approximation was obtained from an appropriate set of Legendre orthonormal polynomials in input variables, whereas the basis for an SCE approximation, given a degree $p$ and a knot sequence of element size $h$, was generated from the Cholesky factorization of the spline moment matrix.  From the uniform distribution, the spline moment matrix was constructed analytically.  All knot sequences are ($p+1$)-open and consist of uniformly spaced distinct knots with even and/or odd numbers of elements, depending on the example.  The PCE and SCE coefficients, which are one-, two-, and four-dimensional integrals, were calculated exactly.

Define, for the first two examples, two approximation errors in the variances,
\begin{equation}
e_{p,h} :=
\displaystyle
\frac{\left|\operatorname{var}[y(\mathbf{X})] - \operatorname{var}[y_{p,h}(\mathbf{X})]\right|}
{\operatorname{var}[y(\mathbf{X})]}
~~\text{and}~~
e_{p} :=
\displaystyle
\frac{\left|\operatorname{var}[y(\mathbf{X})] - \operatorname{var}[y_{p}(\mathbf{X})]\right|}
{\operatorname{var}[y(\mathbf{X})]},
\label{7.1}
\end{equation}
committed by the SCE approximation $y_{p,h}(\mathbf{X}):=y_{p,\boldsymbol{\xi}}(X)$ or $y_{(p,p),\{\boldsymbol{\xi},\boldsymbol{\xi}\}}(X_1,X_2)$
and the PCE approximation $y_{p}(\mathbf{X}):=y_{p}(X)$ or $y_{(p,p)}(X_1,X_2)$, respectively, of $y(\mathbf{X})$.  The exact variance $\operatorname{var}[y(\mathbf{X})]$ was obtained analytically, whereas the SCE variance $\operatorname{var}[y_{p,h}(\mathbf{X})]$ and PCE variance $\operatorname{var}[y_{p}(\mathbf{X})]$ were also determined analytically from \eqref{6.26} and similar formula, respectively. Therefore, all approximation errors were calculated exactly.

\subsection{Example 1: three univariate functions}
Consider a family of three functions of a real-valued, uniformly distributed random variable $X$ over $[-1,1]$:
\begin{equation}
y(X) =
\begin{cases}
\sin(3 \pi X), &       \text{(smooth, oscillatory)}, \\
\exp(-3 |X|),  &       \text{(nonsmooth)}, \\
\Phi(20 X),    &       \text{(nearly discontinuous)}.
\end{cases}
\label{7.2}
\end{equation}
Here, $\Phi(u)=(1/\sqrt{2\pi})\int_{-\infty}^u \exp(-\xi^2/2) d \xi$ is the cumulative probability distribution function of a Gaussian random variable with \emph{zero} mean and \emph{unit} variance.  From top to bottom, \eqref{7.2} comprises oscillatory yet smooth, non-differentiable, and nearly discontinuous functions that are progressively more difficult to approximate by polynomials.

\begin{figure}[b!]
\begin{centering}
\includegraphics[scale=0.53]{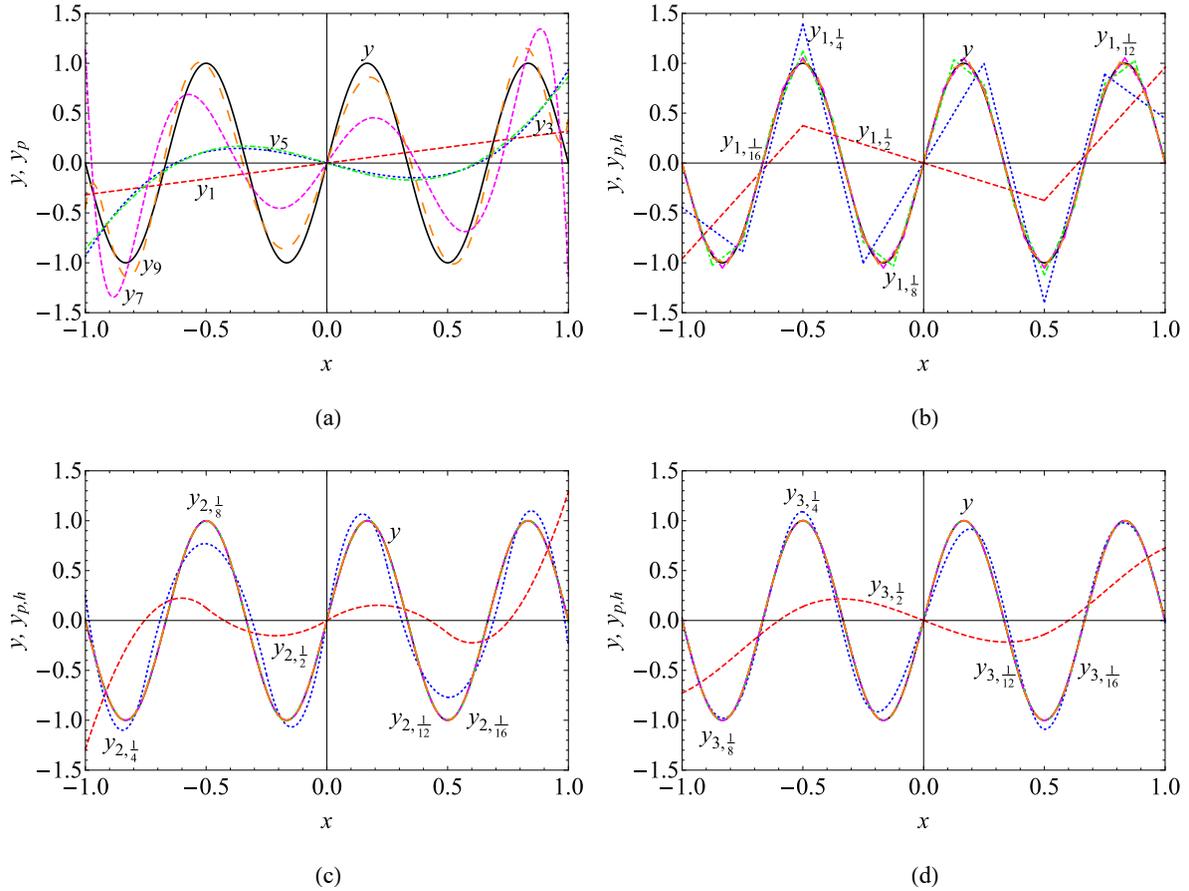}
\par\end{centering}
\caption{Oscillatory function: $y(x)=\sin(3 \pi x)$;
(a) PCE approximations for $p=1,3,5,7,9$;
(b) linear SCE approximations for $h=1/2, 1/4, 1/8, 1/12, 1/16$;
(c) quadratic SCE approximations for $h=1/2, 1/4, 1/8, 1/12, 1/16$;
(d) cubic SCE approximations for $h=1/2, 1/4, 1/8, 1/12, 1/16$.}
\label{fig2}
\end{figure}

The knot sequences for the oscillatory function include simple knots and consist of even numbers of elements with varying element sizes: $h=1/2, 1/4, 1/8, 1/12, 1/16$. For the nonsmooth and nearly discontinuous functions, however, the knot sequences comprise both even and odd numbers of elements, producing the following element sizes: $h=2/5, 2/9, 2/17, 2/25, 2/33$ for odd numbers of elements; and $h=1/2, 1/4, 1/8, 1/12, 1/16$ for even numbers of elements.  The odd numbers of elements are relevant when the location of the point where the function is non-differentiable or nearly discontinuous is unknown.  However, if the aforementioned point is known, then it is possible to employ even numbers of elements by deploying knot(s) at that point as well.  In the latter case, double knots (multiplicity of two for $p=2$) were placed for the non-differentiable function, whereas a single knot was assigned for the nearly discontinuous function.

\begin{figure}[tb!]
\begin{centering}
\includegraphics[scale=0.53]{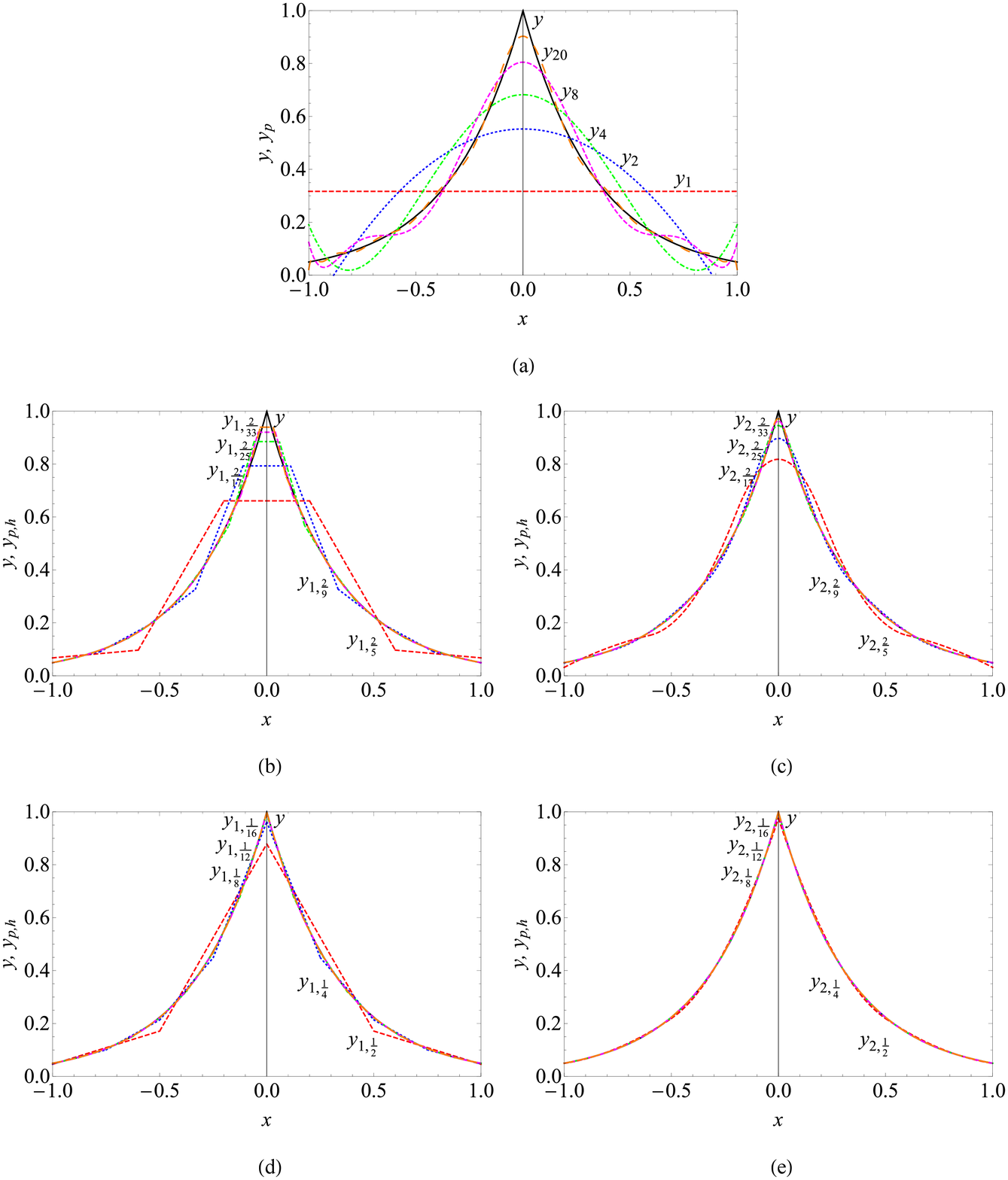}
\par\end{centering}
\caption{Nonsmooth function: $y(x)=\exp(-3 |x|)$;
(a) PCE approximations for $p=1,2,4,8,20$;
(b) linear SCE approximations for $h=2/5, 2/9, 2/17, 2/25, 2/33$ (odd);
(c) quadratic SCE approximations for $h=2/5, 2/9, 2/17, 2/25, 2/33$ (odd);
(d) linear SCE approximations for $h=1/2, 1/4, 1/8, 1/12, 1/16$ (even);
(e) quadratic SCE approximations for $h=1/2, 1/4, 1/8, 1/12, 1/16$ (even).}
\label{fig3}
\end{figure}

Figures \ref{fig2}, \ref{fig3}, and \ref{fig4} depict the comparisons of PCE and SCE approximations for the oscillatory, nonsmooth, and nearly discontinuous functions, respectively.  For the oscillatory function, the PCE approximation improves with $p$ as shown in Figure \ref{fig2}(a), but at the cost of the 9th-degree approximation to be fairly acceptable.  Such requirement becomes stringent for the nonsmooth [Figure \ref{fig3}(a)] or nearly discontinuous [Figure \ref{fig4}(a)] functions, where 20th- or 21st-degree PCE approximations are warranted.  In contrast, the SCE approximations for the oscillatory function, exhibited in Figure \ref{fig2}(b), look satisfactory, if not great, even for a linear spline ($p=1$), as long as the mesh is adequately fine ($h \le 1/8$). For $p=2$ or 3 and $h \le 1/8$, any distinction between an SCE approximation and actual function in Figure \ref{fig2}(c) or Figure \ref{fig2}(d) is indiscernible to the naked eye.

\begin{figure}[tb!]
\begin{centering}
\includegraphics[scale=0.53]{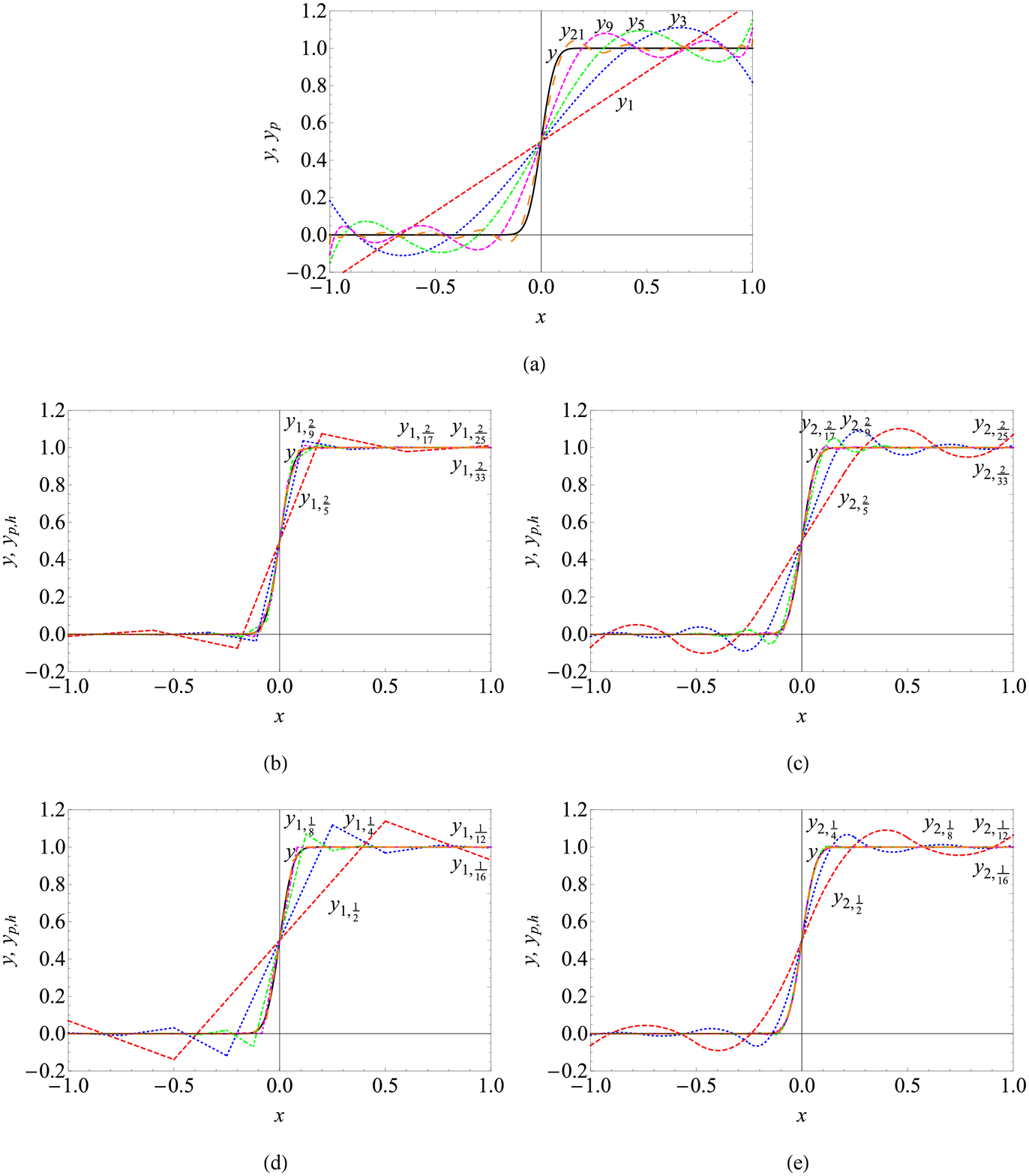}
\par\end{centering}
\caption{Nearly discontinuous function: $y(x)=\Phi(20 x)$;
(a) PCE approximations for $p=1,3,5,9,21$;
(b) linear SCE approximations for $h=2/5, 2/9, 2/17, 2/25, 2/33$ (odd);
(c) quadratic SCE approximations for $h=2/5, 2/9, 2/17, 2/25, 2/33$ (odd);
(d) linear SCE approximations for $h=1/2, 1/4, 1/8, 1/12, 1/16$ (even);
(e) quadratic SCE approximations for $h=1/2, 1/4, 1/8, 1/12, 1/16$ (even).}
\label{fig4}
\end{figure}

\begin{table}
\caption{Relative errors in the variances of three univariate functions by PCE and SCE approximations.}
\begin{centering}
\begin{tabular}{c}
\hline
{\scriptsize{}(a) smooth, oscillatory function: $y(X)=\sin(3\pi X)$}\tabularnewline
{\scriptsize{}}%
\begin{tabular}{ccccccc}
\hline
 &  &  &  & \multicolumn{3}{c}{{\scriptsize{}$e_{p,h}$}}\tabularnewline
\cline{5-7}
{\scriptsize{}$p$} & {\scriptsize{}$e_{p}$} &  & {\scriptsize{}$h$} & {\scriptsize{}$p=1$} & {\scriptsize{}$p=2$} & {\scriptsize{}$p=3$}\tabularnewline
\cline{1-2} \cline{4-7}
{\scriptsize{}1} & {\scriptsize{}0.932453} &  & {\scriptsize{}$1/2$} & {\scriptsize{}0.719505} & {\scriptsize{}0.755454} & {\scriptsize{}0.815119}\tabularnewline
{\scriptsize{}3} & {\scriptsize{}0.823578} &  & {\scriptsize{}$1/4$} & {\scriptsize{}0.0936391} & {\scriptsize{}0.0404491} & {\scriptsize{}0.0107727}\tabularnewline
{\scriptsize{}5} & {\scriptsize{}0.822617} &  & {\scriptsize{}$1/8$} & {\scriptsize{}$3.56392\times10^{-3}$} & {\scriptsize{}$1.5349\times10^{-4}$} & {\scriptsize{}$8.20356\times10^{-6}$}\tabularnewline
{\scriptsize{}7} & {\scriptsize{}0.292768} &  & {\scriptsize{}$1/12$} & {\scriptsize{}$6.05231\times10^{-4}$} & {\scriptsize{}$9.8752\times10^{-6}$} & {\scriptsize{}$1.95148\times10^{-7}$}\tabularnewline
{\scriptsize{}9} & {\scriptsize{}0.0279709} &  & {\scriptsize{}$1/16$} & {\scriptsize{}$1.80817\times10^{-4}$} & {\scriptsize{}$1.56899\times10^{-6}$} & {\scriptsize{}$1.6\times10^{-8}$}\tabularnewline
\hline
\end{tabular}\tabularnewline
\noalign{\vskip0.4cm}
{\scriptsize{}(b) nonsmooth function: $y(X)=\exp(-3|X|)$}\tabularnewline
{\scriptsize{}}%
\begin{tabular}{cccccccccc}
\hline
 &  &  &  & \multicolumn{2}{c}{{\scriptsize{}$e_{p,h}$ (odd no. of elements)}} &  &  & \multicolumn{2}{c}{{\scriptsize{}$e_{p,h}$ (even no. of elements)}}\tabularnewline
\cline{5-6} \cline{9-10}
{\scriptsize{}$p$} & {\scriptsize{}$e_{p}$} &  & {\scriptsize{}$h$} & {\scriptsize{}$p=1$} & {\scriptsize{}$p=2$} &  & {\scriptsize{}$h$} & {\scriptsize{}$p=1$} & {\scriptsize{}$p=2$}\tabularnewline
\cline{1-2} \cline{4-6} \cline{8-10}
{\scriptsize{}1} & {\scriptsize{}1} &  & {\scriptsize{}$2/5$} & {\scriptsize{}0.122349} & {\scriptsize{}0.0212933} &  & {\scriptsize{}$1/2$} & {\scriptsize{}0.0167023} & {\scriptsize{}$4.96606\times10^{-4}$}\tabularnewline
{\scriptsize{}2} & {\scriptsize{}0.325922} &  & {\scriptsize{}$2/9$} & {\scriptsize{}0.026662} & {\scriptsize{}$3.74539\times10^{-3}$} &  & {\scriptsize{}$1/4$} & {\scriptsize{}$1.13075\times10^{-3}$} & {\scriptsize{}$9.37131\times10^{-6}$}\tabularnewline
{\scriptsize{}4} & {\scriptsize{}0.124885} &  & {\scriptsize{}$2/17$} & {\scriptsize{}$4.42555\times10^{-3}$} & {\scriptsize{}$5.48633\times10^{-4}$} &  & {\scriptsize{}$1/8$} & {\scriptsize{}$7.00943\times10^{-5}$} & {\scriptsize{}$1.76989\times10^{-7}$}\tabularnewline
{\scriptsize{}8} & {\scriptsize{}0.0301413} &  & {\scriptsize{}$2/25$} & {\scriptsize{}$1.43968\times10^{-3}$} & {\scriptsize{}$1.71708\times10^{-4}$} &  & {\scriptsize{}$1/12$} & {\scriptsize{}$1.377\times10^{-5}$} & {\scriptsize{}$1.68832\times10^{-8}$}\tabularnewline
{\scriptsize{}20} & {\scriptsize{}0.0037569} &  & {\scriptsize{}$2/33$} & {\scriptsize{}$6.36083\times10^{-4}$} & {\scriptsize{}$7.45087\times10^{-5}$} &  & {\scriptsize{}$1/16$} & {\scriptsize{}$4.34601\times10^{-6}$} & {\scriptsize{}$3.14213\times10^{-9}$}\tabularnewline
\hline
\end{tabular}\tabularnewline
\noalign{\vskip0.4cm}
{\scriptsize{}(c) nearly discontinuous function: $y(X)=\Phi(20 X)$}\tabularnewline
{\scriptsize{}}%
\begin{tabular}{cccccccccc}
\hline
 &  &  &  & \multicolumn{2}{c}{{\scriptsize{}$e_{p,h}$ (odd no. of elements)}} &  &  & \multicolumn{2}{c}{{\scriptsize{}$e_{p,h}$ (even no. of elements)}}\tabularnewline
\cline{5-6} \cline{9-10}
{\scriptsize{}$p$} & {\scriptsize{}$e_{p}$} &  & {\scriptsize{}$h$} & {\scriptsize{}$p=1$} & {\scriptsize{}$p=2$} &  & {\scriptsize{}$h$} & {\scriptsize{}$p=1$} & {\scriptsize{}$p=2$}\tabularnewline
\cline{1-2} \cline{4-6} \cline{8-10}
{\scriptsize{}1} & {\scriptsize{}0.209125} &  & {\scriptsize{}$2/5$} & {\scriptsize{}0.0198118} & {\scriptsize{}0.0574063} &  & {\scriptsize{}$1/2$} & {\scriptsize{}0.0983968} & {\scriptsize{}0.0308548}\tabularnewline
{\scriptsize{}3} & {\scriptsize{}0.0966401} &  & {\scriptsize{}$2/9$} & {\scriptsize{}$2.59428\times10^{-3}$} & {\scriptsize{}0.0184093} &  & {\scriptsize{}$1/4$} & {\scriptsize{}0.0299556} & {\scriptsize{}$5.54174\times10^{-3}$}\tabularnewline
{\scriptsize{}5} & {\scriptsize{}0.0543964} &  & {\scriptsize{}$2/17$} & {\scriptsize{}$2.19365\times10^{-4}$} & {\scriptsize{}$2.05094\times10^{-3}$} &  & {\scriptsize{}$1/8$} & {\scriptsize{}$3.89483\times10^{-3}$} & {\scriptsize{}$5.25409\times10^{-5}$}\tabularnewline
{\scriptsize{}9} & {\scriptsize{}0.0212929} &  & {\scriptsize{}$2/25$} & {\scriptsize{}$1.82967\times10^{-4}$} & {\scriptsize{}$1.81128\times10^{-4}$} &  & {\scriptsize{}$1/12$} & {\scriptsize{}$4.98215\times10^{-4}$} & {\scriptsize{}$2.10786\times10^{-5}$}\tabularnewline
{\scriptsize{}21} & {\scriptsize{}0.0017763} &  & {\scriptsize{}$2/33$} & {\scriptsize{}$7.0312\times10^{-5}$} & {\scriptsize{}$1.50297\times10^{-5}$} &  & {\scriptsize{}$1/16$} & {\scriptsize{}$9.23004\times10^{-5}$} & {\scriptsize{}$1.036\times10^{-5}$}\tabularnewline
\hline
\end{tabular}\tabularnewline
\end{tabular}
\par\end{centering}{\footnotesize \par}
\label{table1}
\end{table}

For the nonsmooth and nearly discontinuous functions, there are two sets of linear ($p=1$) and quadratic ($p=2$) SCE approximations, obtained separately for odd and even numbers of elements; they are displayed in Figures \ref{fig3} and \ref{fig4}.  According to Figures \ref{fig3}(b) and \ref{fig3}(d), the approximation quality of linear SCE approximations for the nonsmooth function is visibly better when there are even numbers of elements, as expected.  The same observation holds for quadratic SCE approximations, where even a much coarser mesh produces excellent approximation for even numbers of elements.  The SCE results for the nearly discontinuous function are qualitatively the same.  However, there are still some oscillations in SCE approximations when the mesh is too coarse, pointing to the Gibb's type phenomenon commonly observed in polynomial-based approximations. Zhang and Martin \cite{zhang97} reported such behavior for a cubic spline approximation of the Heaviside function and found that the oscillation near discontinuity never goes away for a uniform knot sequence.  Clearly, a better, if not optimal, selection of knot sequences is required.

Finally, Table \ref{table1} presents the errors $e_{p,h}$ and $e_{p}$ in the variances of all three functions, obtained using SCE and PCE approximations, respectively, for various chosen degrees and knot sequences.  Clearly, the SCE approximation commits much lower errors than does the PCE approximation for the same degree $p$.  To attain an accurate approximation using splines, one is not interested in large values of $p$. Instead, the motivation is to keep $p$ fixed to a low value, but increase (decrease) the number of knots (element size).  Indeed, Table 1 demonstrates that a low-degree SCE approximation with an adequate mesh is capable of producing significantly more accurate estimates of the variance than the PCE approximation even when its degree of expansion is excessively large.  All approximations errors reported in Table 1 are consistent with the plots displayed in Figures \ref{fig2} through \ref{fig4}.

\subsection{Example 2: solution of a stochastic ODE}
The second example involves a stochastic boundary-value problem, described by the ODE
\begin{equation}
\displaystyle
-\frac{d}{d\xi} \left( \exp(|X_1|) \frac{d}{d\xi}y(\xi;X_1,X_2)\right)=\exp(|X_2|),~0\le \xi \le 1,
~y(\xi;X_1,X_2) \in \mathbb{R},
\label{7.3}
\end{equation}
with boundary conditions
\[
y(0;X_1,X_2)=0,~
\exp(|X_1|) \frac{dy}{d\xi}(1;X_1,X_2)=1.
\]
Here, $X_1$ and $X_2$ are two real-valued, independent, and identically distributed random variables, each following a uniform distribution over $[-1,1]$.  Originally studied by the author \cite{rahman18b}, the ODE is slightly modified here by introducing the absolute-value function, thus producing a nonsmooth solution.

A direct integration of \eqref{7.3} yields the exact solution:
\begin{equation}
y(\xi;X_1,X_2) =
\frac{1}{\exp(|X_1|)}\left[\xi +\left(\xi - \frac{\xi^2}{2} \right)\exp(|X_2|) \right].
\label{7.4}
\end{equation}
Therefore, the first two raw moments of $y(\xi;X_1,X_2)$ can be obtained easily.  For instance, at $\xi=1$, the two moments of $y(1;X_1,X_2)$, denoted briefly as $y(X_1,X_2)$, are
\[
\mathbb{E}[y(X_1,X_2)] = \displaystyle
\frac{1}{e}
\left[1 + \displaystyle \frac{1}{2}(e-1) \right](e-1) \approx 1.1752,
\]
\[
\mathbb{E}[y^2(X_1,X_2)] = \displaystyle
\frac{1}{16 e^2}
(e^2 + 8 e - 1)(e^2-1) \approx 1.52048.
\]
The exact solutions were used to benchmark the approximate results from SCE and PCE approximations.

\begin{figure}[htbp]
\begin{centering}
\includegraphics[scale=0.8]{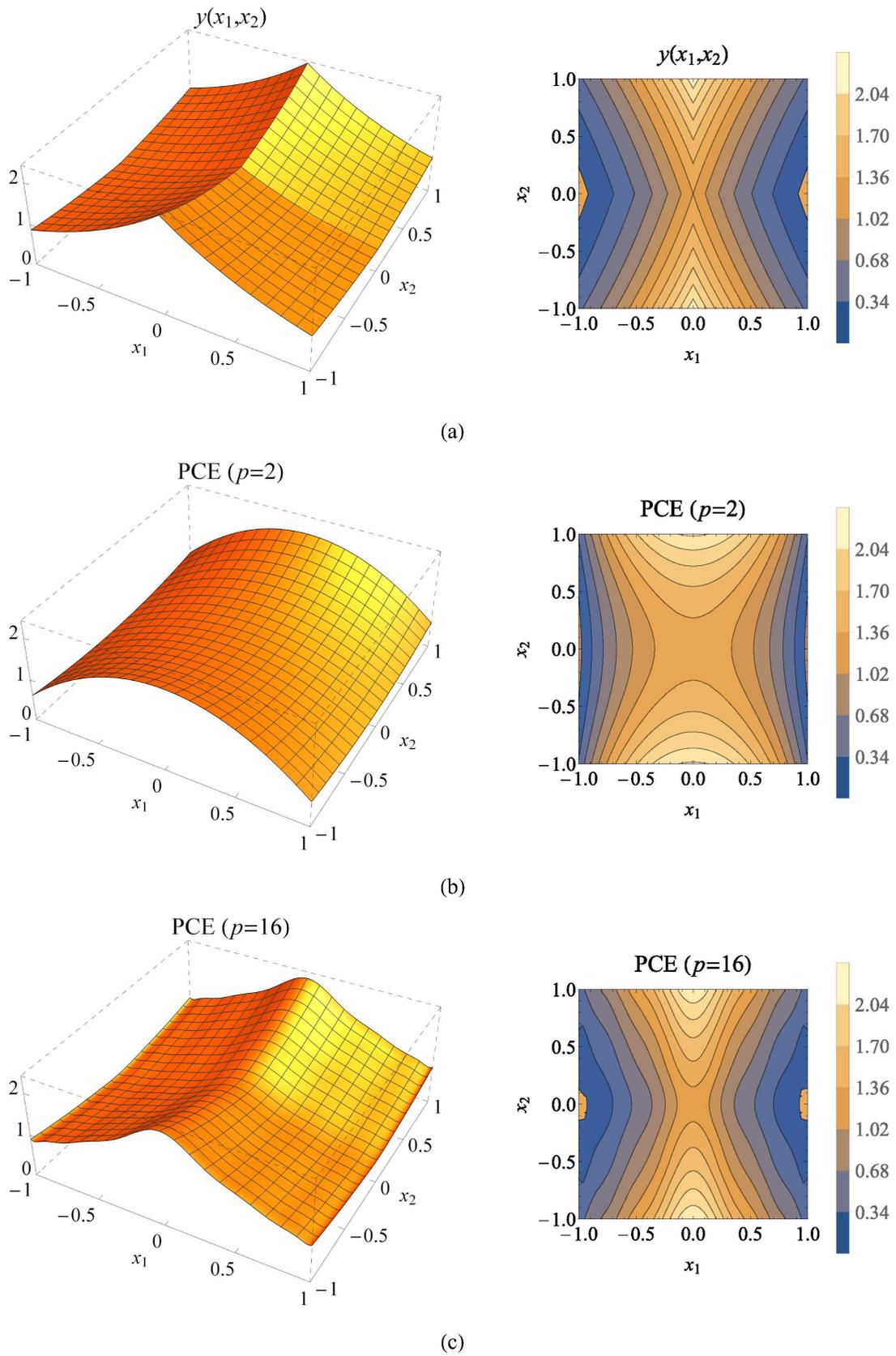}
\par\end{centering}
\caption{Three-dimensional and contour plots of the exact and two PCE solutions of ODE;
(a) exact solution $y(x_1,x_2)$;
(b) second-order PCE approximation;
(c) 16th-order PCE approximation.}
\label{fig5}
\end{figure}

\begin{figure}[htbp]
\begin{centering}
\includegraphics[scale=0.8]{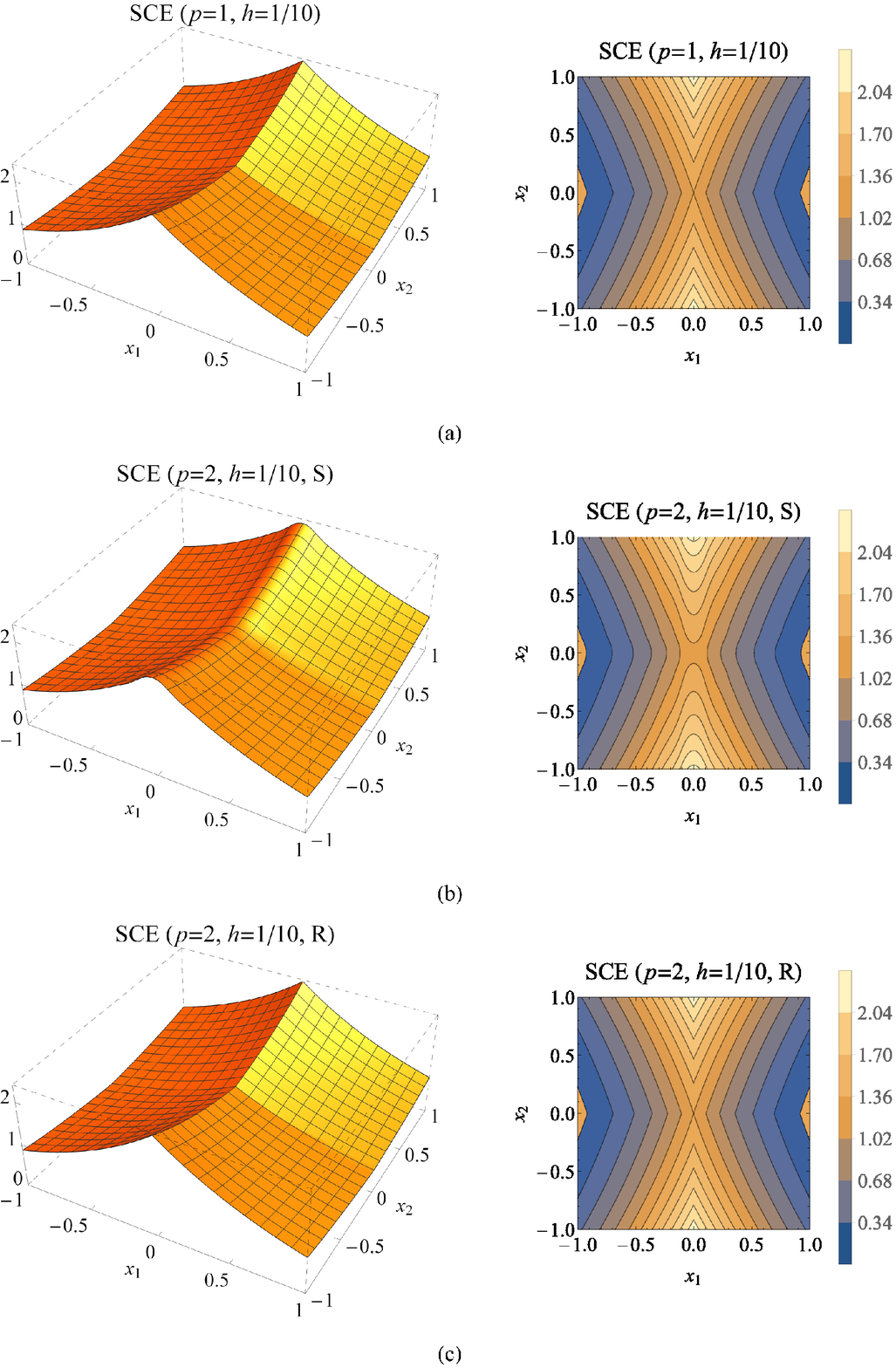}
\par\end{centering}
\caption{Three-dimensional and contour plots of three SCE solutions of ODE;
(a) linear SCE approximation for $h=1/10$;
(b) quadratic SCE approximation for $h=1/10$ and simple (``S") knots;
(c) quadratic SCE approximation for $h=1/10$ and a repeated (``R") central knot.}
\label{fig6}
\end{figure}

Figures \ref{fig5} and \ref{fig6} display three-dimensional (left) and contour (right) plots of the exact function $y(x_1,x_2)$ and several approximations from PCE and SCE. Because of the absolute-value function, the exact solution is saddle-shaped with slope discontinuities at the center, as shown in Figure \ref{fig5}(a).  The second-order PCE approximation exhibited in Figure 5(b) commits a variance error of $e_2=0.116812$ and is clearly inadequate. The 16th-order PCE approximation in Figure \ref{fig5}(c) shows some improvement by reducing the error to $e_{16}=2.26714\times 10^{-3}$, but not to an extent expected from such an impractically high expansion order.

In contrast, the linear ($p=1$) SCE approximation in Figure \ref{fig6}(a), obtained for an even number of elements with an element size of $h=1/10$, matches the exact function extremely well, producing a variance error of $e_{1,1/10}=1.86149\times 10^{-6}$.  The quadratic ($p=2$) SCE approximation in Figure \ref{fig6}(b), generated using the same mesh, yields an error of $e_{2,1/10}=3.54972\times 10^{-4}$, and is better than the 16th-order PCE approximation yet inferior to that in Figure \ref{fig6}(a).  This apparent anomaly of a linear SCE approximation producing a better result than a quadratic SCE approximation can be explained by examining the knot sequences used.  Due to even numbers of elements, there exists a central knot in each coordinate direction for both cases of $p=1$ and $p=2$.  However, for $p=2$, the first-order derivatives are continuous across the central knot in both directions.  This is why the quadratic SCE approximation is smoother than the linear SCE approximation or the exact function.  However, as $y(x_1,x_2)$ is not differentiable at the central knot, the linear approximation performs better than the quadratic approximation. However, if the central knot is repeated (multiplicity of two) in the knot sequences, the quadratic SCE approximation, shown in Figure \ref{fig6}(c), is even better than the linear SCE approximation, resulting in an error of $e_{2,1/10}=4.0056\times 10^{-10}$. Having said so, such manipulations of the knot sequences are not possible in general if the locations of slope discontinuities are not known \emph{a priori}.  In this case, the quadratic SCE approximation in Figure \ref{fig6}(b) is perhaps more realistic and the result of the linear SCE approximation should be deemed fortuitous for this specific problem.

\subsection{Example 3: a nonsmooth function of four variables}
In the final example, consider a nonsmooth function
\begin{equation}
y(\mathbf{X}) =
\displaystyle
\prod_{i=1}^4
\dfrac{|4 X_i - 2|^{b_i} + a_i}{1 + a_i},~a_i,b_i \in \mathbb{R},~i=1,4,
\label{7.5}
\end{equation}
of four independent random variables $X_i$, $i=1,2,3,4$, each of which is uniformly distributed over $[0,1]$.  The function parameters are as follows: $a_1=0$, $a_2=1$, $a_3=2$, $a_4=4$; $b_1=b_2=b_3=b_4=3/5$.   Clearly, $y$ is a non-differentiable function where the exponent $b_i$ controls its nonlinearity. Compared with $b_i=1$, the smaller the value of the exponent, the more nonlinear the function becomes in the $i$th coordinate direction.  This type of function, especially with \emph{unit} exponents, has been used for global sensitivity analysis \cite{saltelli95}.

Figures \ref{fig7}(a) and  \ref{fig7}(b) depict the probability distribution functions of $y(\mathbf{X})$ calculated by three methods: (1) crude MCS; (2) second-, fourth-, and eight-order PCE approximations; and (3) quadratic SCE approximations with three element sizes: $h=1/2$, $h=1/4$, and $h=1/8$. In SCE calculations, there are even numbers of elements for the chosen meshes with repeated central knots ($x_k=0.5$) in each coordinate direction.  Although the basis functions and corresponding expansion coefficients of SCE and PCE approximations were calculated exactly, there is no analytical means to determine their probability distributions.  Instead, the PCE and SCE approximations once built were re-sampled to generate their associated distribution functions.  The sample size for both crude MCS and re-sampling is 10,000, which should be adequate for examining the tail probabilistic characteristics up to a probability of $10^{-3}$.  Compared with the MCS result, the convergence of probability distributions by the SCE approximations in Figure \ref{fig7}(b) is markedly faster than that by the PCE approximations in \ref{fig7}(a).  It appears that low-order SCE approximations also yield more accurate estimates of the probability distributions than a high-order PCE approximation for nonsmooth functions.

\begin{figure}[htbp!]
\begin{centering}
\includegraphics[scale=0.75]{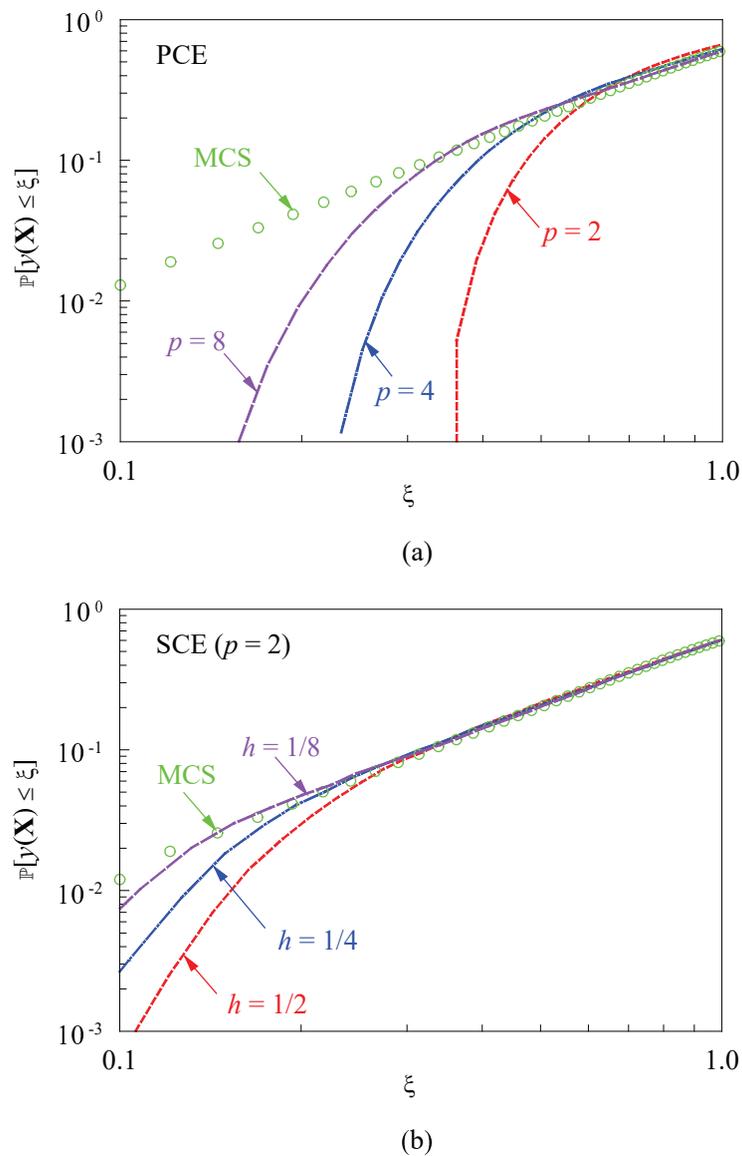}
\par\end{centering}
\caption{Probability distributions of $y(\mathbf{X})$ calculated by three distinct methods;
(a) crude MCS and several PCE approximations;
(b) crude MCS and several SCE approximations.}
\label{fig7}
\end{figure}

\section{Discussion}
While the paper is aimed at fundamental mathematical development of SCE, a brief deliberation on the practical significance of the work is justified.  First, the success of SCE is dependent on its effective implementation for UQ analysis of a general computational model.  For more realistic problems not considered here, the expansion coefficients of SCE approximations cannot be calculated exactly.  In this regard, computationally efficient methods or techniques for estimating the expansion coefficients are direly needed. Given the proliferation of the coefficients, the importance of such a need cannot be overstated. Methods, such as dimension-reduction techniques \cite{xu04} and sparse-grid quadrature \cite{gerstner98}, including a few regression-based approaches used in the PCE community, come to mind.  Some of these methods, when appropriately adapted, may potentially aid in calculating the SCE coefficients economically.

Second, the SCE approximation proposed is designed to account for locally prominent and highly nonlinear stochastic responses, including discontinuity and nonsmoothness, emanating from multiple failure modes of complex systems. On the contrary, if the response is smooth and moderately nonlinear, then existing PCE equipped with globally supported basis is adequate.  In the latter case, there is no significant advantage of an SCE approximations over a PCE approximation.

Third, and more importantly, the use of tensor-product structure to form multivariate B-splines is not always suitable.  Indeed, for high-dimensional UQ problems, tensor-product expansions in the context of SCE or PCE approximations will require an astronomically large number of terms or coefficients, succumbing to the curse of dimensionality. Therefore, developments of alternative computational methods capable of exploiting low effective dimensions of high-dimensional functions, \emph{ \`{a} la} dimensional decomposition methods \cite{rahman18}, are desirable.

These topics are subjects of current research in the author's group.

\section{Conclusion}
A new chaos expansion, namely, SCE of a square-integrable random variable, comprising measure-consistent multivariate orthonormal B-splines in independent random variables, is unveiled.  Under prescribed assumptions, a whitening transformation is proposed to decorrelate univariate B-splines in each coordinate direction into their orthonormal version.  The transformed set of B-splines was proved to form a basis of a general spline space comprising splines of specified degree and knot sequence.  Through a tensor-product structure, multivariate orthonormal B-splines were constructed, spanning the space of multivariate splines of specified degrees and knot sequences in all coordinate directions.  The result is an expansion of a general $L^2$-function with respect to measure-consistent multivariate orthonormal B-splines.  Compared with the existing PCE, SCE, rooted in compactly supported B-splines, deals with locally prominent stochastic responses in a more proficient manner. The approximation quality of the expansion was demonstrated in terms of the modulus of smoothness of the function, leading to the mean-square convergence of SCE to the correct limit. The weaker modes of convergence, such as those in probability and in distribution, follow readily. The optimality of SCE, including deriving PCE as a special case of SCE, was demonstrated.  Analytical formulae akin to those found in the PCE literature are proposed to calculate the mean and variance of an SCE approximation for a general output variable in terms of the expansion coefficients. Numerical results obtained for one-, two-, and four-dimensional UQ problems entailing oscillatory, nonsmooth, and nearly discontinuous functions indicate that a low-order SCE approximation with an adequate mesh is capable of producing a substantially more accurate estimates of the output variance and probability distribution than a PCE with an overly large order of approximation.

\section*{{\small Acknowledgments}}
The author thanks two anonymous reviewers and the associate editor for providing a number of helpful  comments on an earlier draft of the paper.

\bibliographystyle{siamplain}
\bibliography{sce_references}

\end{document}